\documentclass[11pt, reqno]{amsart}
\usepackage{amssymb,amsmath,amsthm}
\usepackage{color}
\usepackage{enumitem}
\usepackage{fullpage}
\usepackage{url}
\usepackage[Symbol]{upgreek}
\usepackage[unicode,pdfusetitle]{hyperref}
\usepackage[mathscr]{euscript}
\usepackage{mathtools}

\usepackage{tikz}
\usetikzlibrary{decorations.pathmorphing, patterns, shapes}
\usepackage{tikz-cd}

\newtheorem{theorem}{Theorem}
\newtheorem*{theorem*}{Theorem}

\newtheorem*{question*}{Question}

\newtheorem*{conjecture*}{Conjecture}

\newtheorem*{convention*}{Convention}

\newtheorem*{assumption*}{Assumption}
\newtheorem{corollary}[theorem]{Corollary}
\newtheorem*{corollary*}{Corollary}

\newtheorem*{remark*}{Remark}
\newtheorem{proposition}[theorem]{Proposition}
\newtheorem*{proposition*}{Proposition}
\newtheorem{lemma}[theorem]{Lemma}
\newtheorem*{lemma*}{Lemma}
\newtheorem{fact}[theorem]{Fact}
\newtheorem*{fact*}{Fact}
\newtheorem{theoremA}{Theorem}

\theoremstyle{definition}

\newtheorem*{definition*}{Definition}

\newtheorem*{example*}{Example}

\numberwithin{theorem}{section}
\numberwithin{equation}{section}

\DeclareMathOperator{\an}{an}

\DeclareMathOperator{\nl}{nl}
\DeclareMathOperator{\nt}{nt}

\DeclareMathOperator{\res}{res}

\DeclareMathOperator{\rc}{rc}

\DeclareMathOperator{\Th}{Th}

\newcommand{\N}{\mathbb{N}}

\newcommand{\Q}{\mathbb{Q}}
\newcommand{\R}{\mathbb{R}}

\newcommand{\T}{\mathbb{T}}
\newcommand{\Z}{\mathbb{Z}}

\newcommand{\cA}{\mathcal A}

\newcommand{\cL}{\mathcal L}

\newcommand{\cO}{\mathcal O}

\newcommand{\cR}{\mathcal R}

\newcommand{\upl}{\uplambda}
\newcommand{\Upl}{\Uplambda}
\newcommand{\upo}{\upomega}
\newcommand{\Upo}{\Upomega}

\newcommand{\inv}{^{-1}}

\newcommand{\No}{\mathbf{No}}
\newcommand{\ee}{\mathrm{e}}

\newcommand{\pow}{_{\operatorname{pow}}}
\newcommand{\Cpow}{_{C\text{-}\!\operatorname{pow}}}
\newcommand{\Rpow}{_{\mathbb{R}\text{-}\!\operatorname{pow}}}
\newcommand{\dsmall}{_{\operatorname{pow},\operatorname{sm}}}
\newcommand{\dlarge}{_{\operatorname{pow},\operatorname{lg}}}
\newcommand{\wpow}{_{\upo,\operatorname{pow}}}
\newcommand{\Hp}{H\pow}

\renewcommand{\preceq}{\preccurlyeq}
\renewcommand{\succeq}{\succcurlyeq}
\renewcommand{\geq}{\geqslant}
\renewcommand{\leq}{\leqslant}

\renewcommand{\epsilon}{\varepsilon}

\renewcommand{\d}{\operatorname{d}}

\renewcommand{\rc}{\operatorname{rc}}

\DeclareFontFamily{OMS}{smallo}{}
\DeclareFontShape{OMS}{smallo}{m}{n}{<->s*[.65]cmsy10}{}
\DeclareSymbolFont{smallo@m}{OMS}{smallo}{m}{n}
\DeclareMathSymbol{\smallo}{\mathord}{smallo@m}{79}

\DeclareFontFamily{U}{fsy}{}
\DeclareFontShape{U}{fsy}{m}{n}{<->s*[.9]psyr}{}
\DeclareSymbolFont{der@m}{U}{fsy}{m}{n}
\DeclareMathSymbol{\der}{\mathord}{der@m}{182}
\DeclareMathSymbol{\derdelta}{\mathord}{der@m}{100}

\author{Elliot Kaplan}
\email{elliot.kaplan@math.su.se}
\address{Matematiska institutionen, Stockholms universitet, Stockholm, Sweden}
\title{Constant power maps on Hardy fields and transseries}
\date{\today}
\begin{document}

\maketitle
\begin{abstract}
Let $\mathbb{T}$ be the differential field of logarithmic-exponential transseries.
We consider the expansion of $\mathbb{T}$ by the binary map that sends a positive transseries $f$ and a real number $r$ to the transseries $f^r$, together with a constant symbol for the real number $\ee$.
Building on recent work of Aschenbrenner, van den Dries, and van der Hoeven, we show that this expansion is model complete, and we give an axiomatization of its theory that is effective relative to the theory of the real exponential field.
We show that maximal Hardy fields, equipped with the same map $(f,r)\mapsto f^r$, enjoy the same theory as $\mathbb{T}$, and we use this to establish a transfer theorem between Hardy fields and transseries.
\end{abstract}
\setcounter{tocdepth}{1}
\tableofcontents

\section*{Introduction}
A \textbf{Hardy field} $H$ is a field of germs of real-valued unary functions at $+\infty$ that is closed under differentiation:\ if the germ of $f$ belongs to $H$, then $f$ must be eventually differentiable and the germ of $f'$ must also belong to $H$.
These axioms rule out any sort of oscillatory behavior; consequently, any Hardy field is a \emph{totally ordered} differential field.

The quintessential example of a Hardy field is $H^{LE}$, the field of germs of Hardy's \emph{logarithmico-exponential functions}.
These are the functions $f\colon \R\to \R$ built from the identity function $x$ by applying exponentials, logarithms, and algebraic operations.
This example predates and inspired the general definition, due to Bourbaki.
The study of Hardy fields and their Hardy field extensions was greatly advanced in the last quarter of the 20th century by Rosenlicht, Boshernitzan, and others.
In the 1990s, Hardy fields became intimately linked with o-minimality, stemming from the observation that an expansion $\cR$ of the real field is o-minimal if and only if the collection $H(\cR)$ of germs of unary $\cR$-definable functions at $+\infty$ forms a Hardy field.

By Zorn's lemma, any Hardy field is contained in one that is \emph{maximal} under inclusion.
The study of Hardy fields has taken a great leap forward with the recent work of Aschenbrenner, van den Dries, and van der Hoeven, who showed that all maximal Hardy fields have the same first-order theory (as differential fields)~\cite{ADH24}.
When considered as \emph{valued} differential fields, with (existentially definable) valuation ring consisting of elements bounded by the germ of a constant function, the theory of maximal Hardy fields is even model complete.
This theory can be effectively axiomatized; hence, it is decidable.

This theory has another model of interest---the differential field $\T$ of logarithmic-exponential transseries.
This field of generalized power series over $\R$ was introduced independently by \'{E}calle~\cite{Eca92} in his solution to the Dulac conjecture and by Dahn and G\"{o}ring~\cite{DG87} in their work on Tarski's problem on real exponentiation.
The elementary theory of $\T$ (and the elementary theory of maximal Hardy fields) is axiomatized by one of two completions of $T^{\nl}$---the theory of $\upo$-free newtonian Liouville closed $H$-fields.
Introduced by Aschenbrenner and van den Dries~\cite{AD02}, $H$-fields form a class of ordered valued differential fields that includes $\T$, all differential subfields of $\T$ containing $\R$, and any Hardy field containing $\R$.
In their book~\cite{ADH17}, Aschenbrenner, van den Dries, and van der Hoeven proved that the theory $T^{\nl}$ is the model companion of the theory of $H$-fields, gave a quantifier elimination result for this theory in an extended language, and established various model-theoretic properties of this theory.
These properties include stable embeddedness of the constant field (as a real closed ordered field) and o-minimality at $+\infty$.
Note that for $\T$ and any Hardy field containing $\R$, the constant field is exactly the field $\R$.

Let $H$ be a Hardy field and suppose that $H$ admits an ordered differential field embedding into $\T$.
Examples abound---$H^{LE}$ admits such an embedding, as does $H(\R_{\an,\exp})$~\cite{DMM97} and, more generally, $H(\R_{\cA,\exp})$ for suitable generalized quasianalytic classes $\cA$~\cite{RSS24}.
Extending the aforementioned results, 
Aschenbrenner, van den Dries, and van der Hoeven showed that any maximal Hardy field extension $M \supseteq H$ has the same first-order theory as $\T$ over $H$ (that is, as a differential field with constant symbols for elements of $H$, where $H$ is identified with its image in $\T$).
Thus, a system of algebraic differential equations, inequalities, and valuation inequalities over $H$ has a solution in $M$ if and only if it has one in $\T$~\cite[Corollary 6]{ADH24}.

For the remainder of this introduction, we fix a maximal Hardy field $M$.
Bourbaki showed that each element in $M$ has an exponential in $M$, and each positive element has a logarithm in $M$~\cite{Bou76}.
Thus, for $f \in M^>$ and $r \in \R$, the germ $f^r \coloneqq \exp(r\log f)$ also belongs to $M^>$.
We call the binary map $(f,r)\mapsto f^r\colon M^>\times \R \to M^>$ the \emph{constant power map} on $M$, and we let $M\pow$ denote the expansion of the valued differential field $M$ by this map together with a constant symbol for the real number $\ee$.
The transseries also admit a natural exponential, so we similarly obtain a constant power map $(f,r)\mapsto f^r\colon \T^>\times \R\to \T^>$.
We let $\T\pow$ denote the expansion of the valued differential field $\T$ by this map together with a constant symbol for $\ee$.

\begin{theoremA}[Corollaries~\ref{cor:TTnlsmall} and~\ref{cor:HardyTnlsmall}]\label{thm:A}
The structures $M\pow$ and $\T\pow$ have the same first-order theory.
This theory is model complete.
\end{theoremA}

Of course, $\T\pow$ is a reduct of $\T_{\exp}$, the expansion of $\T$ by its exponential function.
We show that $\T\pow$ is a proper reduct in Corollary~\ref{cor:nondef} below.
Model completeness for $\T_{\exp}$ remains open.
It is also unknown to us whether the inclusion of a constant symbol for $\ee$ is needed in Theorem~\ref{thm:A}.
Indeed, this question seems to depend on whether the real field with the map $(x,y)\mapsto x^y$ is model complete; see~\cite{Jer22}.

To prove Theorem~\ref{thm:A}, we consider $H$-fields expanded by a map that behaves like the constant power map above.
We call these structures \emph{$\Hp$-fields}, and we show that their theory admits a model companion (under additional assumptions on the field of constants).
Both $\T\pow$ and $M\pow$ are models of this model companion; indeed, this model companion extends the theory of $\Hp$-fields by the same three axioms:\ $\upo$-freeness, newtonianity, and Liouville closedness.
Using this, we can strengthen the first part of Theorem~\ref{thm:A}.

\begin{theoremA}[Theorem~\ref{thm:transfer}]\label{thm:B}
Let $H\subseteq M\pow$ be a Hardy field closed under the constant power map on $M\pow$, and let $\imath\colon H\to \T\pow$ be an ordered differential field embedding that commutes with this map.
Then $M\pow$ and $\T\pow$ have the same first-order theory as differential fields with constant power maps and parameters from $H$.
\end{theoremA}

Thus, the transfer theorem for systems of algebraic differential equations over $H$ can be extended to include \emph{signomial} differential equations (equations involving differential polynomials with real exponents).
Signomials arise naturally in geometric programming and its extensions; see~\cite{Eck80}.
Among signomial ordinary differential equations are $S$-systems, studied by Savageau and others; see~\cite{SV87}.
We note that both $H^{LE}$ and $H(\R_{\cA,\exp})$, with their natural embeddings into $\T$, satisfy the hypotheses on $H$ of the theorem above.

One could, of course, establish a similar transfer theorem by expanding $M$ and $\T$ by unary \emph{real power functions}:\ functions $f\mapsto f^r$ for each $r \in \R$.
As ordered fields with these power functions, $M$ and $\T$ form elementary extensions of the real field with power functions studied by Miller~\cite{Mil94B}.
This expansion of the real field is o-minimal, and so the corresponding expansions of the valued differential fields $M$ and $\T$ by real power functions are \emph{$H_T$-fields} as studied in~\cite{Kap23}; see Subsection~\ref{subsec:nonuni}.
Although model completeness for the expansions of $M$ and $\T$ by real power functions does not follow from our results, it seems fairly straightforward, though cumbersome, to establish by modifying some of the embedding lemmas here and using the results in~\cite{ADH17}.

That said, our framework here is more robust, in that it allows for constant powers to be taken \emph{uniformly}.
We can use this uniformity to analyze families of signomial ODEs, parametrized by real exponents.
As an example of such a family, consider for $\sigma, \rho, \lambda \in \R$ the Emden--Fowler equation
\begin{equation}
\tag{$\operatorname{EF}_{\sigma,\rho,\lambda}$}
\frac{d}{dx}\big(x^\sigma \frac{dy}{dx}\big)+ x^\rho y^\lambda=0.
\end{equation}
The asymptotic behavior of solutions to this equation for various values of the three real parameters was studied by Bellman~\cite{Bel53}.
This family of equations is definable over $H^{LE}$, so Theorem~\ref{thm:B} tells us that ($\operatorname{EF}_{\sigma,\rho,\lambda}$) has a positive Hardy field solution $y(x)$ if and only if it has a positive solution in $\T$.
If we are interested in the structure of the set 
\[
\{(\sigma,\rho,\lambda) \in \R^3:\text{($\operatorname{EF}_{\sigma,\rho,\lambda}$) has a positive solution in $\T$}\},
\]
we need to know what structure is induced on the constant field $\R$.
We note that $\R$ is no longer stably embedded as a pure field, as the real exponential is definable in $\T\pow$.
This turns out to be the only new structure:

\begin{theoremA}[Corollary~\ref{cor:stablyembedded}]\label{thm:C}
In $\T\pow$, the real numbers are stably embedded as an exponential field:\ any definable subset of $\R^n$ is definable in the structure $\R_{\exp}$.
\end{theoremA}

The structure $\R_{\exp}$ was shown to be model complete and o-minimal by Wilkie~\cite{Wil96}.
Given a family of signomial ODEs parametrized by real exponents, such as the family $(\operatorname{EF}_{\sigma,\rho,\lambda})$ above, it follows that whether the equations in this family have a transseries or Hardy field solution only depends on the parameters in a geometrically tame way---e.g., it never depends on the rationality or algebraicity of the exponents.

It would be interesting, though likely quite difficult, to explicitly describe the set of parameters $(\sigma,\rho,\lambda)$ for which $(\operatorname{EF}_{\sigma,\rho,\lambda})$ has a positive transseries solution.
It is not even clear whether this set is semialgebraic or whether it is genuinely more complicated as an $\R_{\exp}$-definable set.
As a toy example, we describe the set 
\[
\{(\sigma,\rho,\lambda) \in \R^3:\text{($\operatorname{EF}_{\sigma,\rho,\lambda}$) has a positive monomial solution}\}, 
\]
where a \emph{positive monomial solution} is a solution of the form $y = cx^r$ with $c \in \R^>$ and $r\in \R$.
Note that this set is still definable in $\T\pow$.
Substituting such a solution into ($\operatorname{EF}_{\sigma,\rho,\lambda}$) yields 
\[
cr(\sigma+r-1)x^{\sigma+r-2}+c^\lambda x^{\rho+\lambda r}=0.
\]
If $\lambda \neq 1$, then such a solution exists if and only if
\[
r= (1-\lambda)\inv(\rho-\sigma+2)\quad\text{and} \quad r(\sigma+r-1)<0.
\]
In this case, $c$ is chosen so that $c^{\lambda-1}=-r(\sigma+r-1)$.
If $\lambda=1$, then such a solution exists if and only if 
\[
\rho=\sigma-2\quad\text{and}\quad \sigma \not\in (-1,3).
\]
In this case, $c$ can be any positive real and $r$ satisfies $r^2+(\sigma-1)r+1=0$.
Thus, the set of parameters for which ($\operatorname{EF}_{\sigma,\rho,\lambda}$) has a positive monomial solution is semialgebraic.

The fact that the exponential on $\R$ is definable in $\T\pow$ tells us that the theory of this structure is at least as complicated as that of $\R_{\exp}$.
This is the only additional complication.

\begin{theoremA}[Corollary~\ref{cor:decidable}]\label{thm:D}
The theory of $\T\pow$ is decidable relative to the theory of $\R_{\exp}$.
\end{theoremA}

Decidability for the theory of $\R_{\exp}$ is intricately linked with transcendental number theory (it is implied by Schanuel's conjecture; see~\cite{MW96}).
It would be interesting to see whether, for a computable $r\in \R$ that is not zero-definable in $\R_{\exp}$, adding the single unary power function $f \mapsto f^r$ yields a proper \emph{unconditionally decidable} expansion of $\T$, in analogy with the decidable proper o-minimal expansions of the real field obtained by Jones and Servi~\cite{JS11}.

Of course, Theorems~\ref{thm:C} and~\ref{thm:D} also hold for $M\pow$ by Theorem~\ref{thm:A}.
At the end of the paper, we study another setting where these theorems hold:\ Conway's field of surreal numbers, equipped with the Berarducci--Mantova derivation and the constant power map induced by the Gonshor--Kruskal exponential.
Building on results from~\cite{ADH19}, we show that this expansion of the surreals, denoted $\No\pow$, has the same first-order theory as $\T\pow$, and that every Hardy field closed under powers admits an embedding into the surreals.

\subsection*{Outline}
We collect necessary preliminary results, mostly from the book~\cite{ADH17}, in Section~\ref{sec:prelim}.
In Section~\ref{sec:powexts}, we investigate power extensions and power closures of $H$-fields; much of the material needed here comes from~\cite[Sections 7 and 8]{AD05}.
In Section~\ref{sec:HP}, we introduce $\Hp$-fields, our primary objects of study.
The focus of this section is our central technical result, Proposition~\ref{prop:mainprop}.
This proposition allows us to convert facts from~\cite{ADH17} on $H$-field embeddings and extensions to corresponding results about $\Hp$-fields.
We apply this proposition throughout Section~\ref{sec:HPext}, where we collect the embedding results required for our main theorems.

In Section~\ref{sec:Tpow}, we prove that the theory of $\Hp$-fields has a model companion, characterize the completions of this model companion, and show that $\T\pow$ is a model.
We also study definable sets in the model companion, proving stable embeddedness of the constant field, local o-minimality, and non-definability of the exponential in $\T\pow$.
In Section~\ref{sec:Hardy}, we prove our results on maximal Hardy fields and the surreal numbers, including Theorem~\ref{thm:B}.

\subsection*{Acknowledgements}
Thanks to Matthias Aschenbrenner and Chris Miller for very helpful conversations about this paper.
Thanks also to Nigel Pynn-Coates for hosting me at the Kurt G\"odel Research Center, where some of this research was conducted.
Other parts of this research were conducted while I was hosted by the Max Planck Institute for Mathematics, and I thank the MPIM for its support and hospitality.
Finally, thanks to the anonymous referee for a number of helpful suggestions, including bringing the need for an additional constant symbol to our attention.
This research was supported in part by the National Science Foundation under Award No.\ DMS-2103240.
\section{Preliminaries}\label{sec:prelim}
We draw heavily from~\cite{ADH17}, and we assume some familiarity with that book.
We use the same notational conventions as~\cite{ADH17}, but we repeat what is needed in this paper.

We let $m$, $n$, and $k$ range over $\N=\{0,1,2,\ldots\}$.
For a ring $R$, we let $R^\times$ denote the invertible elements of $R$.
By ``ordered set'' we mean ``totally ordered set''.
For an ordered abelian group $\Gamma$, we put 
\[
\Gamma^>\coloneqq\{\gamma \in \Gamma:\gamma>0\},\qquad\Gamma^<\coloneqq\{\gamma \in \Gamma:\gamma<0\}.
\]
Given a (possibly infinite) index set $I$, we say that a tuple $(\gamma_i)_{i \in I}$ of elements from $\Gamma$ is \textbf{finitely supported} if the set $\{i \in I:\gamma_i\neq 0\}$ is finite.
We let $\Q\Gamma\coloneqq \Q \otimes_\Z \Gamma$ denote the divisible hull of $\Gamma$, equipped with the unique ordering that makes $\Q\Gamma$ an ordered group extension of $\Gamma$.
For $S \subseteq \Gamma$, we set
\[
S^{\downarrow} \coloneqq \{\gamma\in \Gamma:\gamma\leq \sigma \text{ for some }\sigma \in S\}.
\]

\subsection{Ordered exponential fields, valued fields, and differential fields}

In this subsection, let $K$ be a field of characteristic zero.

\subsubsection*{Ordered exponential fields}
Let $<$ be a field ordering on $K$.
An \textbf{exponential} on $K$ is an ordered group isomorphism $\exp\colon K\to K^>$ from the additive group of $K$ to the multiplicative group of positive elements of $K$.
If $K$ and $L$ are ordered exponential fields (that is, ordered fields equipped with exponentials), then an ordered field embedding $\imath\colon K\to L$ is said to be an \textbf{ordered exponential field embedding} if $\imath(\exp a) = \exp \imath(a)$ for all $a \in K$.
We let $\R_{\exp}$ denote the real exponential field.
This structure is o-minimal, and its first-order theory $\Th(\R_{\exp})$ is model complete in the natural language $\{0,1,+,-,\cdot,<,\exp\}$ of ordered exponential fields by Wilkie's theorem~\cite{Wil96}.

\subsubsection*{Valued fields}

Let $\cO\subseteq K$ be a valuation ring on $K$, so each element of $K$ either belongs to $\cO$ or has a multiplicative inverse in $\cO$.
We denote the unique maximal ideal of $\cO$ by $\smallo$, the residue field $\cO/\smallo$ by $\res(K)$, the value group of $K$ (written additively) by $\Gamma$, and the surjective valuation map $K^\times \to \Gamma$ by $v$.
The group $\Gamma$ is totally ordered, with $va\geq 0\Longleftrightarrow a \in \cO$ and $va> 0\Longleftrightarrow a \in \smallo$.
We extend $v$ to a map $K\to \Gamma_\infty \coloneqq \Gamma \cup \{\infty\}$ by setting $v(0)\coloneqq \infty>\Gamma$.
For $a,b \in K$, we use the following notation:
\[
a\asymp b :\Longleftrightarrow va = vb,\qquad a\preceq b :\Longleftrightarrow va \geq vb,\qquad a\prec b :\Longleftrightarrow va > vb,\qquad a\sim b :\Longleftrightarrow a-b\prec a.
\]
Note that $a\sim b$ if and only if $a,b \neq 0$ and $a/b \in 1+\smallo$.
If $L$ is also a valued field, then we denote its value group, valuation ring, and maximal ideal by $\Gamma_L$, $\cO_L$, and $\smallo_L$, respectively.
Let $L$ be a valued field extension of $K$.
We identify $\Gamma$ with a subgroup of $\Gamma_L$ and $\res(K)$ with a subfield of $\res(L)= \cO_L/\smallo_L$ in the natural way.
We say that $L$ is an \textbf{immediate extension of $K$} if $\Gamma_L = \Gamma$ and $\res(L) = \res(K)$ under this identification.

\subsubsection*{Differential fields}

Let $\der\colon K \to K$ be a derivation on $K$, so $\der(a+b) = \der a+\der b$ and $\der(ab) = a\der b + b\der a$.
We let $C\coloneqq \ker(\der)$ denote the constant field of $K$.
For $a \in K$, we often write $a'$ instead of $\der a$, and if $a \neq 0$, we let $a^\dagger \coloneqq a'/a$ denote the logarithmic derivative of $a$.
We set $(K^\times)^\dagger\coloneqq \{a^\dagger:a \in K^\times\}$.
Then $(K^\times)^\dagger$ is a subgroup of $K$, since $(a^{-1})^\dagger = -a^\dagger$ and $a^\dagger + b^\dagger = (ab)^\dagger$ for $a,b \in K^\times$.
An \textbf{integral} of $a \in K$ is an element $f \in K$ with $f' = a$, and an \textbf{exponential integral} of $a$ is an element $g \in K^\times$ with $g^\dagger = a$.
If $f_1,f_2$ are integrals of $a$ and $g_1,g_2$ are exponential integrals of $a$, then $f_1-f_2\in C$ and $g_1/g_2 \in C^\times$.

If $L$ is also a differential field, then we denote its constant field by $C_L$.
If $y$ is an element of a differential field extension of $K$, then we let $K\langle y\rangle\coloneqq K(y,y',y'',\ldots)$ denote the differential field extension of $K$ generated by $y$.
We say that $y$ is \textbf{$\d$-transcendental over $K$} if the sequence $y,y',y'',\ldots$ is algebraically independent over $K$, and we say that $y$ is \textbf{$\d$-algebraic over $K$} otherwise.

\subsection{$H$-fields}
Let $K$ be an ordered differential field with constant field $C$.
We let $\cO$ be the convex hull of $C$ in $K$, and we view $K$ as an ordered valued differential field with valuation ring $\cO$.

We say that $K$ is an \textbf{$H$-field} if
\begin{enumerate}[label=(H\arabic*)]
\item\label{H1} $f'>0$ for all $f \in K$ with $f>\cO$;
\item\label{H2} $\cO = C+\smallo$.
\end{enumerate}
The class of $H$-fields was introduced by Aschenbrenner and van den Dries~\cite{AD02}.
Any ordered field with trivial derivation is a trivially valued $H$-field, and every $H$-field with nontrivial derivation has a nontrivial valuation ring.

For the remainder of this section, let $K$ be an $H$-field.
An \textbf{$H$-field embedding} is an ordered valued differential field embedding of $K$ into an $H$-field $L$.
By~\ref{H2}, the projection map $\cO \rightarrow \res(K)$ maps $C$ isomorphically onto $\res(K)$.
Consequently, an $H$-field extension $L$ of $K$ is an immediate extension of $K$ if and only if $\Gamma_L = \Gamma$ and $C_L = C$.
By~\ref{H1}, we have for $f \in K^>$ that 
\[
f \succ 1 \Longrightarrow f^\dagger > 0,\qquad f \prec 1 \Longrightarrow f^\dagger = -(f\inv)^\dagger < 0.
\]

\begin{lemma}\label{lem:daggercut}
Let $s \in K$, let $M_1,M_2$ be $H$-field extensions of $K$, and for $i = 1,2$, let $f_i \in M_i^>$ with $f_i^\dagger = s$ and $vf_i \not\in \Q\Gamma$.
Then there is a unique $H$-field embedding $K(f_1)\to M_2$ over $K$ that sends $f_1$ to $f_2$.
\end{lemma}
\begin{proof}
As $vf_i \not\in \Q\Gamma$, each $f_i$ is transcendental over $K$, so there is a field embedding $K(f_1)\to M_2$ over $K$ that sends $f_1$ to $f_2$.
This is a differential field embedding, as $f_1^\dagger = f_2^\dagger$.
To see that this is an $H$-field embedding, it is enough to show that $vf_1$ and $vf_2$ realize the same cut over $\Q\Gamma$.
To see this, take $m,n$ with $n > 0$, let $\gamma \in \Gamma$, and suppose that $nvf_1<m\gamma$ but $m\gamma<nvf_2$.
Take $y \in K^>$ with $vy = \gamma$, so $f_1^n/y^m\succ 1$ and $f_2^n/ y^m \prec 1$.
As $f_1,f_2,y>0$, this gives $(f_1^n/y^m)^\dagger = ns-my^\dagger > 0$, whereas $(f_2^n/y^m)^\dagger = ns-my^\dagger < 0$, a contradiction.
\end{proof}

\subsection{Asymptotic couples and the trichotomy theorem}

As an $H$-field, $K$ is \textbf{$H$-asymptotic}:\ for all $f,g \in \smallo$, we have
\[
f\prec g \Longleftrightarrow f' \prec g',\qquad f\preceq g\Longrightarrow f^\dagger \succeq g^\dagger.
\]
It follows that for $f \in K^\times$ with $f \not\asymp 1$, the values $v(f')$ and $v(f^\dagger)$ only depend on $vf$, so for $\gamma =vf$, we set
\[
\gamma^\dagger \coloneqq v(f^\dagger),\qquad \gamma' \coloneqq v(f') = \gamma+ \gamma^\dagger.
\]
This gives us a map
\[
\psi\colon \Gamma^{\neq}\to \Gamma,\qquad \psi(\gamma) \coloneqq\gamma^\dagger
\]
and, following Rosenlicht~\cite{Ros81}, we call the pair $(\Gamma,\psi)$ the \textbf{asymptotic couple of $K$}.
We have the following important subsets of $\Gamma$:
\[
(\Gamma^<)' \coloneqq \{\gamma':\gamma \in \Gamma^<\},\qquad (\Gamma^>)' \coloneqq \{\gamma':\gamma \in \Gamma^>\},\qquad \Psi \coloneqq \{\gamma^\dagger:\gamma \in \Gamma^{\neq}\}.
\]
It is always the case that $(\Gamma^<)' < (\Gamma^>)'$ and that $\Psi< (\Gamma^>)'$.
For $u \in K$ with $u \asymp 1$, there is $c \in C$ with $u-c \prec 1$, so $v(u') = v(u-c)' \in (\Gamma^>)'$.
Then $v(u^\dagger) \in (\Gamma^>)'$ as well, since $u^\dagger = u'/u \asymp u'$.
We say that $K$ has \textbf{small derivation} if $\der \smallo \subseteq \smallo$ (equivalently, if $(\Gamma^>)'\subseteq \Gamma^>$).
If this is not the case, then $K$ is said to have \textbf{large derivation}.
Note that $K$ has large derivation if and only if $0 \in (\Gamma^>)'$.

If there is $\beta \in \Gamma$ with $\Psi <\beta< (\Gamma^>)'$, then we call $\beta$ a \textbf{gap in $K$}.
There is at most one such $\beta$, and if $\Psi$ has a largest element, then there is no such $\beta$.
If $K$ has trivial valuation (equivalently, if $K$ has trivial derivation, since $K$ is an $H$-field), then the three subsets above are empty, and $0$ is a gap in $K$.
We say that \textbf{$K$ is grounded} if $\Psi$ has a largest element, and we say that \textbf{$K$ is ungrounded} otherwise.
Finally, we say that \textbf{$K$ has asymptotic integration} if $\Gamma = (\Gamma^<)' \cup (\Gamma^>)'$.
If $\beta$ is a gap in $K$ or if $\beta = \max \Psi$, then $\Gamma = (\Gamma^<)' \cup\{\beta\}\cup (\Gamma^>)'$.
We have the following important trichotomy for $H$-asymptotic fields:

\begin{fact}[{\cite[Corollary~9.2.16]{ADH17}}]\label{fact:trich}
If $K$ is an $H$-asymptotic field, then exactly one of the following is true:
\begin{enumerate}
\item $K$ has asymptotic integration;
\item $K$ has a gap;
\item $K$ is grounded.
\end{enumerate}
\end{fact}

\subsection{$\upo$-freeness, newtonianity, and $\d$-algebraic maximality}

A major result in~\cite{ADH17} is that the theory of $H$-fields has a model companion, namely the theory $T^{\nl}$ of $\upo$-free newtonian Liouville closed $H$-fields.
The $H$-field $K$ is \textbf{Liouville closed} if $K$ is real closed and, for each $s \in K$, there are $a \in K$ and $f \in K^\times$ with $a' = f^\dagger = s$.
The axioms ``$\upo$-free'' and ``newtonian'' are more technical, and we will not define these axioms precisely, but we will list some facts about these axioms that will be useful later in this paper.

The property of \textbf{$\upo$-freeness} is a rather subtle axiom that, among other things, rules out the existence of gaps both in $K$ and in various extensions of $K$.
Every $\upo$-free $H$-field has asymptotic integration.
This property is also quite robust; it passes to $\d$-algebraic $H$-field extensions, and it is inherited by certain $H$-subfields:

\begin{fact}[{\cite[Section~11.7 and Theorem~13.6.1]{ADH17}}]\label{fact:dalgupofree}
Suppose that $K$ is $\upo$-free.
If $L$ is a $\d$-algebraic $H$-field extension of $K$, then $L$ is $\upo$-free.
If $E$ is an $H$-subfield of $K$ with $\Gamma_E^<$ cofinal in $\Gamma^<$, then $E$ is $\upo$-free.
\end{fact}

In connection with gaps, let us mention another consequence of $\upo$-freeness:

\begin{lemma}\label{lem:daggergap}
Let $K$ be $\upo$-free and let $f$ be an element in an $H$-field extension of $K$ such that $vf$ is a gap in $K\langle f\rangle$.
Then $K\langle f^\dagger\rangle$ is an immediate extension of $K$.
\end{lemma}
\begin{proof}
This follows from~\cite[Lemmas~11.4.7,~11.5.6, and~11.5.9]{ADH17}.
For readers familiar with $\upl$-sequences from~\cite{ADH17}, we provide a proof.
If $vf$ is a gap in $K\langle f\rangle$, then $-f^\dagger$ is a pseudolimit of a $\upl$-sequence in $K$ by~\cite[Lemmas~11.5.6 and~11.5.9]{ADH17}, but since $K$ is $\upo$-free, any $\upl$-sequence in $K$ is divergent and of $\d$-transcendental type over $K$~\cite[Corollary~13.6.3]{ADH17}, so using~\cite[Lemma~11.4.7]{ADH17}, we get that $K\langle f^\dagger\rangle$ is an immediate extension of $K$.
\end{proof}

The axiom of newtonianity is also rather subtle.
It is, however, quite a strong axiom, especially when coupled with $\upo$-freeness:

\begin{fact}\label{fact:asdmax}
Let $K$ be an $\upo$-free newtonian $H$-field.
Then $K$ is \textbf{asymptotically $\d$-algebraically maximal}, that is, $K$ has no proper immediate $\d$-algebraic $H$-field extensions.
\end{fact}

Fact~\ref{fact:asdmax} was shown under the assumption that $K$ also has divisible value group~\cite[Theorem~14.0.2]{ADH17}, but this divisibility assumption can be removed; see~\cite{PC19}.

\begin{lemma}\label{lem:newtonint}
Suppose that $K$ is $\upo$-free and newtonian and let $s \in K$.
Then there is $a \in K$ with $a' = s$, and if $vs>\Psi$, then there is $f \in K$ with $f^\dagger = s$.
\end{lemma}
\begin{proof}
By~\cite[Proposition 10.2.6]{ADH17}, we can find an element $a$ in an immediate $H$-field extension of $K$ with $a' = s$.
As $K$ is asymptotically $\d$-algebraically maximal, $a$ belongs to $K$.
Now suppose that $vs >\Psi$.
As $K$ has asymptotic integration, $vs \in (\Gamma^>)'$, so~\cite[Lemma 10.4.3]{ADH17} gives $f$ in an immediate $H$-field extension of $K$ with $f\neq 0$ and $f^\dagger = s$.
Again, asymptotic $\d$-algebraic maximality gives $f \in K$.
\end{proof}

If $K$ is $\upo$-free and newtonian, then by Fact~\ref{fact:asdmax}, any $y$ in an immediate $H$-field extension of $K$ must be $\d$-transcendental over $K$.
By~\cite[Lemmas~2.4.2 and~11.4.7]{ADH17}, together with other facts from~\cite[Section~11.4]{ADH17}, the $H$-field extension $K\langle y\rangle$ is completely determined by the cut of $y$ over $K$:

\begin{fact}\label{fact:immextmap}
Let $K$ be $\upo$-free and newtonian and let $K\langle y \rangle$ be an immediate $H$-field extension of $K$.
Let $L$ be an $H$-field extension of $K$ and let $z \in L$ realize the same cut as $y$ over $K$.
Then there is an $H$-field embedding $K\langle y \rangle \to L$ over $K$ that sends $y$ to $z$.
\end{fact}

If $K$ is $\upo$-free, then a \textbf{newtonization of $K$} is by definition an immediate newtonian $H$-field extension of $K$ that embeds over $K$ into any newtonian $H$-field extension of $K$.

\begin{fact}\label{fact:ntexist}
If $K$ is $\upo$-free, then $K$ has a newtonization $K^{\nt}$ that is $\d$-algebraic over $K$.
Any two newtonizations of $K$ are isomorphic over $K$.
\end{fact}

Fact~\ref{fact:ntexist} was shown under the assumption that $\Gamma$ is divisible~\cite[Corollaries~14.3.12 and~14.5.4]{ADH17}, but again, this divisibility assumption can be removed by~\cite{PC19}.

When combined with Liouville closedness, $\upo$-freeness and newtonianity give $\d$-algebraic maximality up to constant field extensions:
\begin{fact}[{\cite[Theorem~16.0.3]{ADH17}}]\label{fact:dmax}
Let $K$ be an $\upo$-free newtonian Liouville closed $H$-field.
Then $K$ has no proper $\d$-algebraic $H$-field extensions with the same constant field.
\end{fact}

If $K$ is $\upo$-free, then a \textbf{Newton-Liouville closure of $K$} is by definition a newtonian Liouville closed $H$-field extension of $K$ that embeds over $K$ into any newtonian Liouville closed $H$-field extension of $K$.

\begin{fact}[{\cite[Corollaries~14.5.10 and~16.2.2]{ADH17}}]\label{fact:nlexist}
Suppose that $K$ is $\upo$-free.
Then $K$ has a Newton-Liouville closure $K^{\nl}$ that is $\d$-algebraic over $K$.
The constant field of $K^{\nl}$ is a real closure of $C$, and any other Newton-Liouville closure of $K$ is isomorphic to $K^{\nl}$ over $K$.
\end{fact}

Combining the previous two facts, we can characterize the Newton-Liouville closures of $K$ as follows:

\begin{lemma}\label{lem:altnl}
Let $K$ be $\upo$-free and let $L$ be a newtonian Liouville closed $H$-field extension of $K$.
Then $L$ is a Newton-Liouville closure of $K$ if and only if $L$ is $\d$-algebraic over $K$ and $C_L$ is algebraic over $C$.
\end{lemma}
\begin{proof}
One direction follows immediately from Fact~\ref{fact:nlexist}.
For the other, suppose that $L$ is $\d$-algebraic over $K$ and $C_L$ is algebraic over $C$, and let $K^{\nl}$ be a Newton-Liouville closure of $K$.
We may identify $K^{\nl}$ with a subfield of $L$; then $L$ is a $\d$-algebraic extension of $K^{\nl}$ and $C_L$, being the real closure of $C$, coincides with the constant field of $K^{\nl}$.
Fact~\ref{fact:dmax} gives $L = K^{\nl}$.
\end{proof}

This gives us some useful cardinality bounds:

\begin{corollary}\label{cor:cardbound1}
Let $K$ be $\upo$-free.
Then any Newton-Liouville closure of $K$ has the same cardinality as $K$.
\end{corollary}
\begin{proof}
Let $K^{\nl}$ be a Newton-Liouville closure of $K$.
The downward L\"owenheim--Skolem theorem gives an elementary $H$-subfield $L\subseteq K^{\nl}$ that contains $K$ with $\vert L\vert = \vert K\vert$.
Then $K^{\nl}$ is a $\d$-algebraic $H$-field extension of $L$ with the same constant field as $L$, so $L = K^{\nl}$ by Fact~\ref{fact:dmax}.
\end{proof}

\begin{corollary}\label{cor:cardbound2}
Let $K$ be $\upo$-free and let $L$ be a $\d$-algebraic $H$-field extension of $K$.
If $C_L$ is algebraic over $C$, then $\vert L\vert = \vert K\vert $.
\end{corollary}
\begin{proof}
By Fact~\ref{fact:dalgupofree}, $L$ is $\upo$-free as well.
Let $L^{\nl}$ be a Newton-Liouville closure of $L$.
Then $L^{\nl}$ is $\d$-algebraic over $K$ and $C_{L^{\nl}}$ is algebraic over $C$, so Lemma~\ref{lem:altnl} gives that $L^{\nl}$ is a Newton-Liouville closure of $K$.
Corollary~\ref{cor:cardbound1} gives $\vert K\vert \leq \vert L\vert \leq \vert L^{\nl}\vert = \vert K\vert $.
\end{proof}

\section{Power extensions and the power closure}\label{sec:powexts}
In this section, let $K$ be an $H$-field with real closed constant field $C$.
Following~\cite{AD05}, we say that $K$ is \textbf{closed under powers} if for all $f \in K^\times$ and all $c \in C$, there is $g \in K^\times$ with $g^\dagger = cf^\dagger$.
Equivalently, $K$ is closed under powers if the additive subgroup $(K^\times)^\dagger$ is a $C$-linear subspace of $K$.

\begin{fact}[{\cite[Lemma~7.4]{AD05}}]\label{fact:gammaCspace}
Suppose that $K$ is closed under powers.
Then $\Gamma$ admits the structure of an ordered $C$-vector space.
Explicitly, for $\gamma \in \Gamma$ and $c \in C$, take $f,g\in K^\times$ with $vf = \gamma$ and $g^\dagger = cf^\dagger$ and set $c\gamma\coloneqq vg$.
This value does not depend on the choices of $f$ and $g$.
\end{fact}

For $g \in K^\times$, we have $g \asymp 1$ if and only if $v(g^\dagger) > \Psi$.
This gives us the following:

\begin{lemma}\label{lem:poweroverPsi}
Suppose that $K$ is closed under powers, let $f_1,\ldots,f_n \in K^\times$, and let $c_1,\ldots,c_n \in C$.
Then
\[
c_1vf_1+ \cdots+c_nvf_n = 0 \Longleftrightarrow v(c_1f_1^\dagger + \cdots + c_nf_n^\dagger)>\Psi.
\]
\end{lemma}
\begin{proof}
Choose for each $i = 1,\ldots,n$ some $g_i \in K$ with $g_i^\dagger =c_if_i^\dagger$ and set $g \coloneqq g_1g_2\cdots g_n$.
Then 
\[
v(c_1f_1^\dagger + \cdots + c_nf_n^\dagger)= v(g^\dagger) >\Psi \Longleftrightarrow c_1vf_1+\cdots +c_nvf_n = vg = 0.\qedhere
\]
\end{proof}

\begin{lemma}\label{lem:uniquesim1}
Suppose that $K$ is closed under powers, let $f \in 1+\smallo$, and let $c \in C$.
Then there is a unique $g \in 1+\smallo$ with $g^\dagger = cf^\dagger$.
\end{lemma}
\begin{proof}
Take $a \in K^\times$ with $a^\dagger = cf^\dagger$.
Then $a \asymp 1$, so we may take $d \in C^\times$ with $a \sim d$.
Put $g \coloneqq a/d$.
Then $g \sim 1$ and $g^\dagger = a^\dagger = cf^\dagger$.
To see that $g$ is unique, let $h \in 1+\smallo$ with $h^\dagger =cf^\dagger = g^\dagger$.
Then 
\[
h/g \in C^\times \cap (1+\smallo) = \{1\},
\]
so $h = g$.
\end{proof}

A \textbf{power extension} of $K$ is an $H$-field extension $L$ of $K$ such that $C_L = C$ and each $a \in L$ is contained in an intermediate subfield $K(t_1,\ldots,t_n) \subseteq L$ where for each $i$, either $t_i$ is algebraic over $K(t_1,\ldots,t_{i-1})$ or $t_i^\dagger = cf^\dagger$ for some $c\in C$ and some nonzero $f \in K(t_1,\ldots,t_{i-1})$.
A \textbf{power closure} of $K$ is a power extension of $K$ that is real closed and closed under powers.
In~\cite[Section~8]{AD05}, Aschenbrenner and van den Dries studied power closures, establishing the following:
\begin{fact}[{\cite[Proposition~8.5]{AD05}}]
$K$ has either exactly one or exactly two power closures up to $K$-isomorphism.
\end{fact}
Moreover, they characterized when $K$ has two distinct power closures:

\begin{fact}[{\cite[Proposition~8.10]{AD05}}]\label{fact:closurechar}
$K$ has two power closures nonisomorphic over $K$ if and only if there are $c_1,\ldots,c_n \in C$ and $f_1,\ldots,f_n \in K^\times$ such that
\[
\Psi<v(c_1f_1^\dagger + \cdots+c_nf_n^\dagger)<(\Gamma^>)'.
\]
In particular, if $K$ is grounded or has asymptotic integration, then $K$ has a unique power closure.
\end{fact}

If $L$ is a power closure of $K$, then $\Gamma_L$ admits the structure of an ordered $C$-vector space, so it makes sense to speak of $C\Gamma$, the $C$-linear span of $\Gamma$, as a subspace of $\Gamma_L$.
In the case that $K$ has a unique power closure, we have a good understanding of the value group and $\Psi$-set of this power closure.

\begin{fact}[{\cite[Corollary~8.11]{AD05}}]\label{fact:psicofinal}
Suppose that $K$ has a unique power closure $K\pow$.
Then $\Gamma_{K\pow} = C\Gamma$.
If $K$ is grounded, then $\max \Psi_{K\pow} = \max \Psi$, and if $K$ is ungrounded, then $\Gamma^>$ is coinitial in $\Gamma_{K\pow}^>$.
Consequently, $\Psi$ is cofinal in $\Psi_{K\pow}$ and a gap in $\Gamma$ remains a gap in $\Gamma_{K\pow}$.
\end{fact}

\begin{lemma}\label{lem:intermediateclosure}
Suppose that $K$ has a unique power closure $K\pow$ and let $M \subseteq K\pow$ be an $H$-field extension of $K$.
Then $K\pow$ is also the unique power closure of $M$.
\end{lemma}
\begin{proof}
Clearly, $K\pow$ is a power closure of $M$, as any $H$-field extension of $M$ that is a power extension of $K$ is also a power extension of $M$.
Thus, we need only show that $M$ has a unique power closure.
If $K$ is grounded, then this follows from Facts~\ref{fact:closurechar} and~\ref{fact:psicofinal}, since $\Psi \subseteq \Psi_M$ is cofinal in $\Psi_{K\pow}$.
Suppose that $K$ is ungrounded, so $\Gamma^> \subseteq \Gamma_M^>$ is coinitial in $\Gamma_{K\pow}^>$.
Let $f_1,\ldots,f_n \in M^>$ and $c_1,\ldots,c_n \in C$ and take $g \in K\pow$ with $g^\dagger = c_1f_1^\dagger + \cdots + c_nf_n^\dagger$.
Either $vg^\dagger \in \Psi_{K\pow}$ or $vg = 0$ and $vg^\dagger = vg' \in (\Gamma_{K\pow}^>)'$.
Since $\Psi_M$ is cofinal in $\Psi_{K\pow}$ and $(\Gamma_M^>)'$ is coinitial in $(\Gamma_{K\pow}^>)'$, we conclude from Fact~\ref{fact:closurechar} that $M$ has a unique power closure.
\end{proof}

It may well happen that $K$ has a unique power closure up to isomorphism but not up to \emph{unique} isomorphism.
However, in the case that $K\pow$ is an immediate extension of $K$ (or, more generally, of the real closure of $K$), then this isomorphism is indeed unique.

\begin{lemma}\label{lem:immediateunique}
Suppose that $K$ has a unique power closure $K\pow$ and suppose that $\Q\Gamma = \Gamma_{K\pow}$.
Then $K\pow$ is unique up to unique $K$-isomorphism:\ if $M$ is an $H$-field extension of $K$ that is closed under powers, then there is a unique $H$-field embedding $K\pow\to M$ over $K$.
\end{lemma}
\begin{proof}
Let $\imath,\jmath\colon K\pow\to M$ be two $H$-field embeddings over $K$.
We need to show that $\imath = \jmath$.
Let $K_0 \coloneqq \{f \in K\pow:\imath(f) = \jmath(f)\}$ be the equalizer of $\imath$ and $\jmath$.
Then 
\begin{enumerate}[label=(\roman*)]
\item $K_0$ is an $H$-field extension of $K$;
\item\label{K0realclosed} $K_0$ is real closed, since $\imath$ and $\jmath$ extend uniquely to the real closure of $K_0$;
\item\label{K0immediate} $K\pow$ is an immediate extension of $K_0$, since $\Gamma_{K_0} = \Gamma_{K\pow}$ by~\ref{K0realclosed}.
\end{enumerate}
Let $f \in K\pow^\times$ with $f^\dagger \in K_0$.
Then $\imath(f^\dagger) = \jmath(f^\dagger)$, so $\imath(f)/\jmath(f) \in C_M^\times$.
Using~\ref{K0immediate}, take $g \in K_0^\times$ with $f \sim g$ and put $u \coloneqq f/g \sim 1$.
We have $\imath(g) = \jmath(g)$, so 
\[
\imath(f)/\jmath(f) = \imath(u)/\jmath(u)\in C_M^\times.
\]
As $\imath(u)\sim\jmath(u) \sim 1$, we conclude that $\imath(u)/\jmath(u) = 1$, so $\imath(f)=\jmath(f)$ and $f \in K_0$.
Thus, $K_0$ has no proper power extensions in $K\pow$, so $K_0 = K\pow$ by~\cite[Lemma~8.2]{AD05}.
\end{proof}

\subsection{Linear independence in $H$-fields}\label{subsec:linind}

As evidenced by the criterion in Fact~\ref{fact:closurechar}, $C$-linear combinations of elements of $K$ play a crucial role in determining the structure of power extensions of $K$.
In the proof of our main technical result, Proposition~\ref{prop:mainprop}, we will need a handful of lemmas about these $C$-linear combinations, which we record below.

\begin{lemma}\label{lem:minimalQindep}
Let $f_1,\ldots,f_n \in K^\times$ and $c_1,\ldots,c_n \in C$.
If $c_n$ is in the $\Q$-linear span of $\{c_1,\ldots,c_{n-1}\}$, then we can find $g_1,\ldots,g_{n-1} \in K^\times$ with
\[
c_1f_1^\dagger + \cdots+c_nf_n^\dagger\asymp c_1g_1^\dagger + \cdots+c_{n-1}g_{n-1}^\dagger.
\]
\end{lemma}
\begin{proof}
Take $m_1,\ldots,m_n \in \Z$ with $m_n \neq 0$ and $m_nc_n = m_1c_1+\cdots+ m_{n-1}c_{n-1}$.
For each $j = 1,\ldots,n-1$, set $g_j \coloneqq f_j^{m_n}f_n^{m_j}$, so we have
\begin{align*}
c_1g_1^\dagger+\cdots+ c_{n-1}g_{n-1}^\dagger &= c_1(m_nf_1^\dagger +m_1f_n^\dagger) + \cdots + c_{n-1}(m_nf_{n-1}^\dagger + m_{n-1}f_n^\dagger)\\
& =m_nc_1f_1^\dagger+\cdots +m_nc_{n-1}f_{n-1}^\dagger+ (m_1c_1+\cdots +m_{n-1}c_{n-1})f_n^\dagger \\
&= m_n(c_1f_1^\dagger + \cdots+c_nf_n^\dagger) \asymp c_1f_1^\dagger + \cdots+c_nf_n^\dagger.\qedhere
\end{align*}
\end{proof}

By the axiom~\ref{H2}, the underlying valued $C$-vector space of $K$ is actually a Hahn space over $C$; see~\cite[Section 2.3]{ADH17}.
The next lemma follows from~\cite[Lemmas~2.3.1 and~4.6.16]{ADH17}, but we include a proof for completeness.

\begin{lemma}\label{lem:hahndecomp}
Let $M$ be an $H$-field extension of $K$ and let $f \in K[C_M]\subseteq M$.
Then $f = d_1g_1+\cdots +d_mg_m$ for some $d_1,\ldots,d_m \in C_M$ and $g_1,\ldots,g_m \in K^\times$ with $vg_1,\ldots,vg_m$ distinct.
\end{lemma}
\begin{proof}
Write $f = c_1f_1+\cdots+ c_nf_n$, with $c_1,\ldots,c_n \in C_M$ and $f_1,\ldots,f_n \in K$.
We prove the lemma by induction on $n$, with the case $n=0$ being trivial.
Suppose $n>0$ and, applying the lemma to $c_1f_1+\cdots+ c_{n-1}f_{n-1}$, take $d_1,\ldots,d_m \in C_M$ and $g_1,\ldots,g_m \in K^\times$ with distinct valuations such that 
\[
c_1f_1+\cdots+ c_{n-1}f_{n-1} = d_1g_1+\cdots +d_mg_m.
\]
We claim that we can find $\tilde{c}_1,\ldots,\tilde{c}_m \in C$ with
\[
v(f_n - \tilde{c}_1g_1- \cdots - \tilde{c}_mg_m) \not\in \{vg_1,\ldots,vg_m\}.
\]
If $vf_n \not\in\{vg_1,\ldots,vg_m\}$, we may take all the $\tilde{c}_i$ to be zero.
Otherwise, we use the fact that if $f_n \asymp g_i$, then $f_n- \tilde{c}g_i \prec g_i$ for some $\tilde{c}\in C$ by~\ref{H2}.
Put $h \coloneqq f_n - \tilde{c}_1g_1- \cdots - \tilde{c}_mg_m$.
Then $vh \not\in \{vg_1,\ldots,vg_m\}$ and 
\[
f = (d_1+c_n \tilde{c}_1)g_1+\cdots + (d_m+c_n \tilde{c}_m)g_m+ c_nh.\qedhere 
\]
\end{proof}

\begin{lemma}\label{lem:indepover}
Let $M$ be an $H$-field extension of $K$ and let $S\subseteq \Gamma_M$.
Let $f_1,\ldots,f_n \in K$ and suppose that $f_1,\ldots,f_n$ are $C$-linearly independent over $\{g \in K:vg>S\}$.
Then $f_1,\ldots,f_n$ are $C_M$-linearly independent over $\{h \in M:vh>S\}$.
\end{lemma}
\begin{proof}
Let $I = \{g \in K:vg>S\}$, so $I$ is a $C$-linear subspace of $K$ and $f_1,\ldots,f_n$ are $C$-linearly independent over $I$.
Let $c_1,\ldots,c_n \in C_M$, not all zero, and let $f = c_1f_1+\cdots + c_nf_n$.
We need to show that $vf \in S^{\downarrow}$.
Since $f \in K[C_M]$, we may use Lemma~\ref{lem:hahndecomp} to take $d_1,\ldots,d_m \in C_M^\times$ and $g_1\succ g_2 \succ\cdots \succ g_m\in K^\times$ with $f =d_1g_1 + \cdots + d_m g_m$.
By~\cite[Lemma~4.6.16]{ADH17}, $K$ and $C_M$ are linearly disjoint over $C$, so we conclude that some nontrivial $C$-linear combination of $f_1,\ldots,f_n$ is contained in the $C$-linear span of $g_1,\ldots,g_m$.
As $f_1,\ldots,f_n$ are $C$-linearly independent over $I$, we conclude that some $g_i \not\in I$.
As $I$ is $\preceq$-downward closed, we get that $g_1 \not\in I$, so $vf = vg_1 \in S^{\downarrow}$.
\end{proof}

\section{$H$-fields with constant power maps}\label{sec:HP}
In this section, we study constant power maps on $H$-fields and prove our main technical result, Proposition~\ref{prop:mainprop}, which allows us to ``upgrade'' various $H$-field extensions to extensions equipped with compatible constant power maps.

Let $K$ be an $H$-field with constant field $C$.
A \textbf{constant power map} on $K$ is a binary function 
\[
(f,c) \mapsto f^c\colon K^>\times C \to K^>
\] 
that satisfies the following axioms (where $f,g$ range over $K^>$ and $c,d$ range over $C$):
\begin{enumerate}[label=(C\arabic*)]
\item\label{C1} $f^c f^d = f^{c+d}$, $(f^c)^d = f^{cd}$, and $f^1 = f$.
\item\label{C2} $(fg)^c = f^cg^c$.
\item\label{C3} If $f \sim 1$, then $f^c \sim 1$.
\item\label{C4} $(f^c)^\dagger = cf^\dagger$.
\item\label{C5} The map $c\mapsto 2^c$ is an ordered group isomorphism $C\to C^>$.
\end{enumerate}
We note that the map sending $(f,c)\in (1+\smallo)\times C$ to the unique element $f^c\in 1+\smallo$ satisfying $(f^c)^\dagger=cf^\dagger$, obtained from Lemma~\ref{lem:uniquesim1}, satisfies~\ref{C1}--\ref{C4}.
We refer to the map $c\mapsto 2^c$ in~\ref{C5} as the \textbf{induced exponential on $C$}, and we consider $C$ as an ordered exponential field with this induced exponential.
The following lemma describes the relationship between constant power maps and exponentials on the constant field.
The proof is routine.

\begin{lemma}\label{lem:induceddetermines}
Let $K$ be an $H$-field and let $\exp$ be an exponential on $C$ with functional inverse $\log$.
Then the map
\[
(d,c)\mapsto \exp(c\log d)\colon C^>\times C\to C^>
\]
satisfies~\textup{\ref{C1}}--\textup{\ref{C4}}.
Suppose in addition that $K$ is equipped with a constant power map $(f,c)\mapsto f^c$.
Then the following are equivalent:
\begin{enumerate}
\item $\exp$ coincides with the induced exponential on $C$.
\item $\exp(1)=2$ and $d^c=\exp(c\log d)$ for all $c\in C$ and $d\in C^>$.
\end{enumerate}
\end{lemma}

Many natural examples of constant power maps arise from exponential functions.
Instead of verifying the axioms~\ref{C1}--\ref{C5} in each of these examples, we show here that any exponential function on an $H$-field that is compatible with the derivation gives rise to a constant power map.

\begin{proposition}\label{prop:exptopow}
Let $K$ be an $H$-field, let $\exp\colon K\to K^>$ be an exponential on $K$, and let $\log$ be the compositional inverse of $\exp$.
Suppose that $\exp$ satisfies the identity $(\exp a)^\dagger = a'$ for $a \in K$.
Then the map $(f,c) \mapsto \exp(c\log f)\colon K^>\times C \to K^>$ is a constant power map on $K$.
\end{proposition}
\begin{proof}
A routine computation shows that $(f,c) \mapsto \exp(c\log f)$ satisfies~\ref{C1} and~\ref{C2}.
The identity $\exp(a)^\dagger = a'$ also ensures that this map satisfies~\ref{C4}, and it tells us that $\exp$ restricts to an ordered group isomorphism $C\to C^>$.
Using that multiplication by $\log 2$ is an ordered group isomorphism $C\to C$, we see that $c \mapsto \exp(c\log 2)\colon C\to C^>$ is an ordered group isomorphism, so~\ref{C5} holds.
For~\ref{C3}, let $f \in K^>$ with $f \sim 1$ and let $c \in C$.
If $c = 0$, then $\exp(c\log f) = 1$.
Suppose $c>0$, and let $\epsilon \in C$ with $0<\epsilon<1$.
We need to show that $1-\epsilon<\exp(c\log f)<1+\epsilon$.
Put 
\[
d_1\coloneqq \exp(c\inv\log(1-\epsilon))\qquad d_2 \coloneqq\exp(c\inv\log(1+\epsilon)),
\]
so $d_1,d_2 \in C$ and $d_1<1<d_2$.
As $f \sim 1$, we have $d_1<f<d_2$, so 
\[
1-\epsilon = \exp(c\log d_1)<\exp(c\log f) < \exp(c\log d_2) = 1+\epsilon.
\]
The case $c<0$ is similar.
\end{proof}

By the remark following~\cite[Lemma~7.1]{AD05}, we can find an exponential on $K$ as in the above proposition whenever $C$ admits an exponential and $K$ is Liouville closed.

We now single out the class of structures that will be our primary objects of study.
An \textbf{$\Hp$-field} is a real closed $H$-field $K$ equipped with a constant power map.
If $K$ is an $\Hp$-field, then $K$ is closed under powers, so $\Gamma$ admits the structure of a $C$-vector space by Fact~\ref{fact:gammaCspace}.
If $L$ is also an $\Hp$-field, then a map $\imath\colon K\to L$ is said to be an \textbf{$\Hp$-field embedding} if $\imath$ is an $H$-field embedding and $\imath(f^c)=\imath(f)^{\imath(c)}$ for all $f\in K^>$ and $c\in C$.
In particular, every $\Hp$-field embedding restricts to an ordered exponential field embedding of the constant fields equipped with their induced exponentials.
The notions of an $\Hp$-field extension and an $\Hp$-subfield are defined analogously.
We record the following simple criterion for recognizing $\Hp$-subfields.

\begin{lemma}\label{lem:Hpsubfield}
Let $K$ be an $\Hp$-field and let $E\supseteq C$ be a real closed $H$-subfield of $K$ that is closed under powers.
Then $E$ is an $\Hp$-subfield of $K$.
\end{lemma}
\begin{proof}
Let $f \in E^>$ and $c \in C$.
Since $E$ is closed under powers, we can find $g \in E^\times$ with $g^\dagger = cf^\dagger = (f^c)^\dagger$, so $g = df^c$ for some $d \in C^\times$.
Since $C \subseteq E$, we get $f^c \in E$.
\end{proof}

\subsection{$\Hp$-field extensions from $H$-field extensions}
In this subsection, let $K$ be an $\Hp$-field.
In axiomatizing the model companion of the theory of $H$-fields, Aschenbrenner, van den Dries, and van der Hoeven establish a great number of extension and embedding results for $H$-fields.
To avoid reproving these results from scratch for $\Hp$-fields, we need a general criterion that tells us when we can ``upgrade'' these extensions and embeddings to also be compatible with a constant power map.
We give this criterion in the following technical proposition, and we list various simplifications in the remainder of this subsection.

\begin{proposition}\label{prop:mainprop}
Let $M$ be an $H$-field extension of $K$ with real closed constant field $C_M$, and let $\exp_M$ be an exponential on $C_M$ extending the induced exponential on $C$.
Let $(a_i)_{i \in I}$ be a tuple of elements in $M^>$ and suppose that
\begin{enumerate}[label=(\roman*)]
\item\label{mainpropi} The tuple $(va_i)_{i \in I}$ spans $\Q\Gamma_M$ as a $\Q$-vector space over $\Gamma$.
\item\label{mainpropii} The tuple $(a_i^\dagger)_{i \in I}$ is $C_M$-linearly independent over the $C_M$-subspace of $M$ generated by $(K^\times)^\dagger$ and $\{h \in M:vh>\Psi_M\}$.
\end{enumerate}
Then we have the following:
\begin{enumerate}
\item $M$ has a unique power closure $M\pow$ up to $M$-isomorphism;
\item There is a constant power map on $M\pow$ making $M\pow$ into an $\Hp$-field extension of $K$ with induced exponential $\exp_M$;
\item This constant power map is unique up to unique $M$-isomorphism:\ if $L$ is an $\Hp$-field extension of $K$ and $\imath\colon M\to L$ is an $H$-field embedding over $K$ that restricts to an ordered exponential field embedding $C_M\to C_L$, then $\imath$ extends to a unique $\Hp$-field embedding $M\pow\to L$.
\end{enumerate}
\end{proposition}
\begin{proof}
First, we show that $M$ has a unique power closure up to $M$-isomorphism, using Fact~\ref{fact:closurechar}.
Let $c_1,\ldots,c_n \in C_M$ and $f_1,\ldots,f_n \in M^\times$, and suppose towards contradiction that 
\[
\Psi_M < v(c_1f_1^\dagger + \cdots+c_nf_n^\dagger)< (\Gamma_M^>)'.
\]
We assume $n$ is minimal, so $c_1,\ldots,c_n$ are $\Q$-linearly independent by Lemma~\ref{lem:minimalQindep}.
By our assumption~\ref{mainpropi}, we can find for each $j = 1,\ldots,n$ a finitely supported tuple of rational numbers $(q_{i,j})_{i \in I}$ and an element $g_j \in K$ such that 
\[
f_j \asymp g_j\prod_{i\in I}a_i^{q_{i,j}}, 
\]
where the product on the right belongs to $M^{\rc}$, the real closure of $M$, so that the rational powers of the $a_i$ make sense.
For each $j$, also take an element $u_j \in M^{\rc}$ with $u_j \asymp 1$ such that
\[
f_j = g_ju_j\prod_{i\in I}a_i^{q_{i,j}}.
\]
Then $u_j^\dagger$ belongs to $M$, as it can be written as the sum of elements in $M$.
For each $i \in I$, let $d_i \coloneqq c_1q_{i,1}+\cdots +c_nq_{i,n}$.
Then the tuple $(d_i)_{i \in I}$ is also finitely supported, and we have
\[
c_1f_1^\dagger + \cdots+c_nf_n^\dagger = (c_1g_1^\dagger+\cdots +c_ng_n^\dagger) + (c_1u_1^\dagger +\cdots +c_nu_n^\dagger) + \sum_{i \in I}d_ia_i^\dagger.
\]
Since $g_1^\dagger,\ldots,g_n^\dagger \in (K^\times)^\dagger$ and since $u_1^\dagger,\ldots,u_n^\dagger$ and $c_1f_1^\dagger + \cdots+c_nf_n^\dagger$ belong to the set $\{h \in M:vh>\Psi_M\}$, we may use our assumption~\ref{mainpropii} to deduce that $d_i = 0$ for all $i$.
As $c_1,\ldots,c_n$ are $\Q$-linearly independent, we must have that $q_{i,j} = 0$ for all $i \in I$ and all $j = 1,\ldots,n$.
Therefore, $f_j = g_ju_j$ for all $j$.
Since $v(u_j^\dagger) \in (\Gamma_M^>)'$ and 
\[
v(c_1f_1^\dagger + \cdots+c_nf_n^\dagger) = v(c_1g_1^\dagger+\cdots +c_ng_n^\dagger + c_1u_1^\dagger +\cdots +c_nu_n^\dagger) < (\Gamma_M^>)',
\]
we have $c_1f_1^\dagger + \cdots+c_nf_n^\dagger \asymp c_1g_1^\dagger + \cdots+c_ng_n^\dagger$.
By replacing $f_j$ by $g_j$ for each $j$, we may arrange that each $f_j$ is in $K^\times$.
Now, since $v(c_1f_1^\dagger + \cdots+c_nf_n^\dagger)>\Psi_M$, we apply Lemma~\ref{lem:indepover} with $f_i^\dagger$ in place of $f_i$ and $\Psi_M$ in place of $S$.
This gives $\tilde{c}_1,\ldots,\tilde{c}_n \in C$, not all zero, with $v(\tilde{c}_1f_1^\dagger + \cdots+\tilde{c}_nf_n^\dagger)>\Psi$.
We may as well assume that $\tilde{c}_n \neq 0$ and, dividing through by $\tilde{c}_n$, we may even assume that $\tilde{c}_n = 1$.
Since $K$ is an $\Hp$-field, we find $h \in K$ with $h^\dagger = \tilde{c}_1f_1^\dagger + \cdots+\tilde{c}_{n-1}f_{n-1}^\dagger+f_n^\dagger$.
Then $v(h^\dagger) \in (\Gamma^>)'$, since $v(h^\dagger)>\Psi$.
We have
\[
c_1f_1^\dagger + \cdots+c_nf_n^\dagger = (c_1-c_n\tilde{c}_1)f_1^\dagger + \cdots + (c_{n-1}- c_n\tilde{c}_{n-1})f_{n-1}^\dagger + c_nh^\dagger.
\]
Since $v(h^\dagger) \in (\Gamma^>)'$ and $v(c_1f_1^\dagger + \cdots+c_nf_n^\dagger)< (\Gamma_M^>)'$, we have
\[
c_1f_1^\dagger + \cdots+c_nf_n^\dagger \asymp (c_1-c_n\tilde{c}_1)f_1^\dagger + \cdots + (c_{n-1}- c_n\tilde{c}_{n-1})f_{n-1}^\dagger.
\]
This contradicts the minimality of $n$.

We have now established that $M$ has a unique power closure $M\pow$ up to $M$-isomorphism.
Now, we turn to constructing a constant power map on $M\pow$.
For each $i\in I$, we put $\alpha_i \coloneqq va_i$.
We also fix a $C$-linear basis $(\beta_j)_{j \in J}$ for $\Gamma$, and we choose for each $j$ an element $b_j \in K^>$ with $vb_j = \beta_j$.
We claim that $(\alpha_i)_{i \in I}$ and $(\beta_j)_{j \in J}$ together form a $C_M$-linear basis for $\Gamma_{M\pow} = C_M\Gamma_M$.
Our assumption that $(\alpha_i)_{i \in I}$ spans $\Q\Gamma_M$ as a $\Q$-vector space over $\Gamma$ ensures that $(\alpha_i)_{i \in I}$ and $(\beta_j)_{j \in J}$ span $\Gamma_{M\pow}$ as a $C_M$-vector space.
For $C_M$-linear independence, let $(c_i)_{i \in I}$ and $(d_j)_{j \in J}$ be finitely supported tuples from $C_M$ and suppose that 
\[
\sum_{i \in I}c_i\alpha_i+ \sum_{j\in J}d_j\beta_j = 0.
\]
By Lemma~\ref{lem:poweroverPsi} (with $M\pow$ in place of $K$), we get that $v\big( \sum_{i \in I}c_ia_i^\dagger + \sum_{j \in J} d_jb_j^\dagger\big)> \Psi_{M\pow}$.
As $\Psi_M$ is cofinal in $\Psi_{M\pow}$, we may use our assumption~\ref{mainpropii} to see that each $c_i = 0$, so $v\big(\sum_{j \in J} d_jb_j^\dagger\big)> \Psi_{M\pow}$.
Applying Lemma~\ref{lem:poweroverPsi}, this time to $K$ and the tuple $(b_j)_{j \in J}$, we see that $(b_j^\dagger)_{j \in J}$ is $C$-linearly independent over $\{g \in K:vg>\Psi\}$.
Applying Lemma~\ref{lem:indepover} with the $b_j^\dagger$ in place of the $f_i$ and with $\Psi_{M\pow}$ in place of $S$, we conclude that each $d_j = 0$ as well.

With this claim established, we will define our constant power map, starting with the powers of the $a_i$ and $b_j$.
Fix a $\Q$-linear basis $(\hat{c}_\eta)_{\eta<\mu}$ for $C$ with $\hat{c}_0 = 1$, and extend this to a $\Q$-linear basis $(\hat{c}_\eta)_{\eta<\nu}$ for $C_M$ (so $\nu\geq \mu$, with equality if and only if $C_M = C$).
For each $i \in I$ and each $0<\eta<\nu$, we fix an element $a_{i,\eta} \in M\pow^>$ with $a_{i,\eta}^\dagger = \hat{c}_\eta a_i^\dagger$.
We extend this to all $\eta<\nu$ by putting $a_{i,0} \coloneqq a_i$.
For each $j \in J$ and each $\mu\leq \eta < \nu$, we fix an element $b_{j,\eta} \in M\pow^>$ with $b_{j,\eta}^\dagger = \hat{c}_\eta b_j^\dagger$.
We extend this to all $\eta<\nu$ by putting $b_{j,\eta} \coloneqq b_j^{\hat{c}_\eta}$ for $\eta<\mu$.
Now, let $c \in C_M$ and write $c = \sum_{\eta<\nu}q_\eta \hat{c}_\eta$ where $(q_\eta)_{\eta<\nu}$ is a finitely supported tuple from $\Q$.
For $i \in I$, we put $a_i^c\coloneqq \prod_{\eta<\nu}a_{i,\eta}^{q_\eta}$.
Likewise, for $j \in J$, we put $b_j^c\coloneqq \prod_{\eta<\nu}b_{j,\eta}^{q_\eta}$.
Clearly, for each $a_i$ we have
\[
a_i^{\hat{c}_\eta} = a_{i,\eta}\text{ for all }\eta,\qquad (a_i^c)^\dagger = ca_i^\dagger,\qquad v (a_i^c) = c\alpha_i,
\]
and likewise for the $b_j$.
Now, we extend this constant power map to all of $M\pow$.
Let $f \in M\pow^>$ and take finitely supported tuples $(c_i)_{i \in I}$ and $(d_j)_{j \in J}$ from $C_M$ with $vf = \sum_{i \in I}c_i\alpha_i+ \sum_{j\in J}d_j\beta_j$.
Then
\[
f = d(1+\epsilon)\prod_{i \in I}a_i^{c_i}\prod_{j \in J}b_j^{d_j}
\]
for some unique $d \in C_M^>$ and $\epsilon \in M\pow$ with $\epsilon \prec 1$.
We put $d^c \coloneqq \exp_M(c\log_Md)$ and, using Lemma~\ref{lem:uniquesim1}, we let $(1+\epsilon)^c$ be the unique element of $1+ \smallo_{M\pow}$ with $((1+\epsilon)^c)^\dagger = c(1+\epsilon)^\dagger$.
We put
\[
f^c \coloneqq d^c(1+\epsilon)^c\prod_{i \in I}a_i^{c_ic}\prod_{j \in J}b_j^{d_jc}.
\]
It is routine to check that this is indeed a constant power map extending that of $K$.
By Lemma~\ref{lem:induceddetermines}, this constant power map makes $M\pow$ into an $\Hp$-field extension of $K$ with induced exponential $\exp_M$.

It remains to show that the claimed universal property holds.
Let $L$ be an $\Hp$-field extension of $K$, let $\imath\colon M\to L$ be an $H$-field embedding over $K$, and suppose that $\imath\vert_{C_M}\colon C_M\to C_L$ is an ordered exponential field embedding.
We note that $(va_{i,\eta})_{i \in I,0<\eta<\nu}$ and $(vb_{j,\eta})_{j \in J,\mu\leq \eta<\nu}$ form a $\Q$-linear basis for $\Gamma_{M\pow}$ over $\Q\Gamma_M$.
Let $M_0$ be the field extension of $M$ generated by $(a_{i,\eta})_{i \in I,0<\eta<\nu}$ and $(b_{j,\eta})_{j \in J,\mu\leq \eta<\nu}$.
As $a_{i,\eta}^\dagger = \hat{c}_\eta a_i^\dagger \in M$ and $b_{j,\eta}^\dagger = \hat{c}_\eta b_j^\dagger\in M$ for all $i \in I$, $j \in J$, and $\eta< \nu$, we may iteratively apply Lemma~\ref{lem:daggercut} to find an $H$-field embedding $M_0\to L$ that sends $a_{i,\eta}$ to $\imath(a_i)^{\imath(\hat{c}_\eta)}$ and $b_{j,\eta}$ to $\imath(b_j)^{\imath(\hat{c}_\eta)}$ for each $i,j,\eta$.
By Lemma~\ref{lem:intermediateclosure}, $M\pow$ is the unique power closure of $M_0$.
As $\Q\Gamma_{M_0} = \Gamma_{M\pow}$, our embedding $M_0\to L$ extends uniquely to an $H$-field embedding $\jmath\colon M\pow\to L$ by Lemma~\ref{lem:immediateunique}.
We claim that $\jmath$ is an $\Hp$-field embedding.
Let $c= \sum_{\eta<\nu}q_\eta\hat{c}_\eta \in C_M$.
For $i \in I$, we have
\[
\jmath(a_i^c)= \prod_{\eta<\nu}\jmath(a_{i,\eta})^{q_\eta} = \prod_{\eta<\nu}\imath(a_i)^{q_\eta\imath(\hat{c}_\eta)} = \imath(a_i)^{\imath(c)} = \jmath(a_i)^{\jmath(c)}.
\]
Likewise, $\jmath(b_j^c) = \jmath(b_j)^{\jmath(c)}$.
Let $\epsilon \in M\pow$ with $\epsilon \prec 1$.
We have
\[
\jmath\big((1+\epsilon)^c\big)^\dagger= \jmath\big(c(1+\epsilon)^\dagger\big) = \jmath(c)\jmath(1+\epsilon)^\dagger,
\]
so Lemma~\ref{lem:uniquesim1} gives that $\jmath\big((1+\epsilon)^c\big) = \jmath(1+\epsilon)^{\jmath(c)}$.
Using Lemma~\ref{lem:induceddetermines} and our assumption that $\imath$ commutes with the induced exponentials on $C_M$ and $C_L$, we get that $\jmath(d^c) = \jmath(d)^{\jmath(c)}$ for all $d \in C_M^>$.
As shown above, any $f \in M\pow^>$ may be decomposed as a product
\[
f = d(1+\epsilon)\prod_{i \in I}a_i^{c_i}\prod_{j \in J}b_j^{d_j},
\]
where $d \in C_M^>$, $\epsilon \prec 1$, and $(c_i)_{i \in I}$ and $(d_j)_{j \in J}$ are from $C_M$.
From this, it follows that $\jmath(f^c)= \jmath(f)^{\jmath(c)}$, so $\jmath$ is indeed an $\Hp$-field embedding.
As for the uniqueness of $\jmath$, we note that \emph{any} $\Hp$-field embedding $M\pow\to L$ extending $\imath$ must send $a_{i,\eta} = a_i^{\hat{c}_\eta}$ to $\imath(a_i)^{\imath(\hat{c}_\eta)}$ for each $i \in I$ and $\eta<\nu$; likewise with the $b_{j,\eta}$.
The restriction of $\jmath$ to $M_0$ is uniquely determined by this requirement, so $\jmath$ is unique as well by Lemma~\ref{lem:immediateunique}.
\end{proof}

In many applications of the above proposition, the extension $M$ has the same constant field as $K$.
In this case, we can use that $(K^\times)^\dagger$ and $\{h \in M:vh>\Psi_M\}$ are both $C$-subspaces of $M$ and that $\exp_M$ is necessarily the induced exponential on $C$ to simplify the statement of the proposition:

\begin{corollary}\label{cor:maincor}
Let $M$ be an $H$-field extension of $K$ with $C_M = C$, let $(a_i)_{i \in I}$ be a tuple of elements in $M^>$, and suppose that
\begin{enumerate}[label=(\roman*)]
\item The tuple $(va_i)_{i \in I}$ spans $\Q\Gamma_M$ as a $\Q$-vector space over $\Gamma$.
\item The tuple $(a_i^\dagger)_{i \in I}$ is $C$-linearly independent over $(K^\times)^\dagger+\{h \in M:vh>\Psi_M\}\subseteq M$.
\end{enumerate}
Then $M$ has a unique power closure $M\pow$, and there is a unique constant power map on $M\pow$ that makes it an $\Hp$-field extension of $K$.
If $L$ is an $\Hp$-field extension of $K$, then any $H$-field embedding $M\to L$ over $K$ extends uniquely to an $\Hp$-field embedding $M\pow\to L$.
\end{corollary}

The statement of the proposition simplifies even further when the elements of the tuple $(v(a_i^\dagger))_{i \in I}$ are distinct.
This situation occurs a few times in the remainder of the paper, so we record it here:

\begin{corollary}\label{cor:maincor2}
Let $M$ be an $H$-field extension of $K$ with $C_M = C$.
Let $(\alpha_i)_{i \in I}$ be a tuple of nonzero elements in $\Gamma_M$ and suppose that $(\alpha_i)_{i \in I}$ spans $\Q\Gamma_M$ as a $\Q$-vector space over $\Gamma$ and that $(\alpha_i^\dagger)_{i \in I}$ is a tuple of distinct elements of $\Gamma_M \setminus \Psi$.
Then $M$ has a unique power closure $M\pow$, and there is a unique constant power map on $M\pow$ that makes it an $\Hp$-field extension of $K$.
If $L$ is an $\Hp$-field extension of $K$, then any $H$-field embedding $M\to L$ over $K$ extends uniquely to an $\Hp$-field embedding $M\pow\to L$.
\end{corollary}
\begin{proof}
Take for each $i \in I$ an element $a_i \in M^>$ with $va_i = \alpha_i$.
By Corollary~\ref{cor:maincor}, we need only show that $(a_i^\dagger)_{i \in I}$ is $C$-linearly independent over 
\[
(K^\times)^\dagger + \{h \in M :vh >\Psi_M\}.
\]
To see this, let $(c_i)_{i \in I}$ be a finitely supported tuple from $C$, not all zero, and let $g \in K^\times$.
As the $\alpha_i^\dagger$ are distinct from one another and from $vg^\dagger$, we have
\[
v\big(g^\dagger + \sum_{i \in I}c_ia_i^\dagger\big) = \min\big(\{\alpha_i^\dagger:c_i \neq 0\}\cup\{vg^\dagger\}\big) \in \Psi_M.\qedhere
\]
\end{proof}

Another frequently occurring situation is the case that $M$ is an immediate extension of $K$.
\begin{corollary}\label{cor:extensioncor1}
Let $M$ be an immediate $H$-field extension of $K$.
Then $M$ has a unique power closure $M\pow$.
This power closure is an immediate extension of $K$, and there is a unique constant power map on $M\pow$ that makes it an $\Hp$-field extension of $K$.
If $L$ is an $\Hp$-field extension of $K$, then any $H$-field embedding $M\to L$ over $K$ extends uniquely to an $\Hp$-field embedding $M\pow\to L$.
\end{corollary}
\begin{proof}
Apply Corollary~\ref{cor:maincor}, taking $(a_i)_{i \in I}$ to be the empty tuple.
Fact~\ref{fact:psicofinal} gives $\Gamma_{M\pow} = C\Gamma_M = C\Gamma = \Gamma$, so $M\pow$ is an immediate extension of $K$.
\end{proof}

The final situation we record here is the case where the value group of $M$ is generated by one element over the value group of $K$.
\begin{corollary}\label{cor:extensioncor2}
Let $M$ be an $H$-field extension of $K$ with $C_M = C$ and $\Gamma_M = \Gamma \oplus \Z\alpha$ for some $\alpha \in \Gamma_M^>$.
Then $M$ has a unique power closure $M\pow$, and there is a unique constant power map on $M\pow$ that makes it an $\Hp$-field extension of $K$.
If $L$ is an $\Hp$-field extension of $K$, then any $H$-field embedding $M\to L$ over $K$ extends uniquely to an $\Hp$-field embedding $M\pow\to L$.
\end{corollary}
\begin{proof}
Fix $a \in M^>$ with $va = \alpha$.
By Corollary~\ref{cor:maincor}, it is enough to show that $a^\dagger \not\in (K^\times)^\dagger+\{h \in M:vh>\Psi_M\}$.
Let $g \in K^\times$.
If $v(a^\dagger + g^\dagger) = v((ag)^\dagger) > \Psi_M$, then $ag \asymp 1$, so $a\asymp g\inv$, contradicting that $va\not\in \Gamma$.
\end{proof}

\section{Extensions and embeddings of $\Hp$-fields}\label{sec:HPext}
In this section, we collect the embedding results we need to prove our main theorem.
We do this by combining various $H$-field embedding results from~\cite{ADH17} with the tools established in the previous section on extending $H$-field embeddings to $\Hp$-field embeddings.
For the remainder of this section, $K$ is an $\Hp$-field.
\subsection{Differentially algebraic $\Hp$-field extensions}
In this subsection, we handle the various differentially algebraic extensions of our $\Hp$-field $K$.
We begin with the newtonization and the Newton-Liouville closure.

\begin{proposition}\label{prop:newtonization}
Suppose that $K$ is $\upo$-free.
Then there is a unique constant power map on $K^{\nt}$, the newtonization of $K$, that makes it an immediate $\Hp$-field extension of $K$.
If $L$ is a newtonian $\Hp$-field extension of $K$, then there is an $\Hp$-field embedding $K^{\nt}\to L$ over $K$.
\end{proposition}
\begin{proof}
As $K^{\nt}$ is an immediate $H$-field extension of $K$, we may use Corollary~\ref{cor:extensioncor1} along with the universal property of the newtonization to see that the proposition holds for $(K^{\nt})\pow$.
Since $K^{\nt}$ is asymptotically $\d$-algebraically maximal by Fact~\ref{fact:asdmax} and $(K^{\nt})\pow$ is an immediate $\d$-algebraic extension of $K^{\nt}$ by Corollary~\ref{cor:extensioncor1}, we have $K^{\nt} = (K^{\nt})\pow$.
\end{proof}

\begin{lemma}\label{lem:expint}
Let $s \in K$ with $s < 0$, and suppose that $v(s-a^\dagger)\in \Psi^\downarrow$ for all $a\in K^\times$.
Then $K$ has an $\Hp$-field extension $K(f)\pow$ with $f>0$, $f^\dagger = s$, and $C_{K(f)\pow} = C$.
This extension is unique up to $K$-isomorphism:\ given an $\Hp$-field extension $L$ of $K$ and $g \in L^>$ with $g^\dagger = s$, there is a unique $\Hp$-field embedding $K(f)\pow\to L$ that sends $f$ to $g$.
\end{lemma}
\begin{proof}
Lemma~10.5.20 in~\cite{ADH17} gives an $H$-field extension $K(f)$ of $K$ with 
\[
f>0,\qquad f^\dagger = s,\qquad \Gamma_{K(f)} = \Gamma \oplus \Z vf,\qquad C_{K(f)} = C,
\]
and the desired universal property.
The lemma follows from Corollary~\ref{cor:extensioncor2}.
\end{proof}

\begin{proposition}\label{prop:nlclosure}
Suppose that $K$ is $\upo$-free.
Then $K$ has a newtonian Liouville closed $\Hp$-field extension $K^{\nl}$ that embeds over $K$ into any newtonian Liouville closed $\Hp$-field extension of $K$.
This extension $K^{\nl}$ is $\d$-algebraic over $K$, with $C_{K^{\nl}} = C$.
\end{proposition}
\begin{proof}
Define a \emph{Newton-Liouville tower on $K$} to be a tower $(K_\mu)_{\mu\leq\nu}$ of $\d$-algebraic $\Hp$-field extensions of $K$ such that
\begin{enumerate}[label=(\roman*)]
\item $K_0 = K$ and $K_\lambda = \bigcup_{\mu<\lambda}K_\mu$ whenever $0<\lambda\leq \nu$ is a limit ordinal;
\item if $\mu<\nu$ and $K_\mu$ is not newtonian, then $K_{\mu+1} = K_\mu^{\nt}$ with the constant power map given in Proposition~\ref{prop:newtonization};
\item if $\mu<\nu$, $K_\mu$ is newtonian, and there is $s \in K_\mu^<$ with $v(s-a^\dagger)\in \Psi^\downarrow$ for all $a \in K_\mu^\times$, then $K_{\mu+1} = K_\mu(f)\pow$ as described in Lemma~\ref{lem:expint}, so $f>0$ and $f^\dagger = s$.
\end{enumerate}
If $(K_\mu)_{\mu\leq\nu}$ is a Newton-Liouville tower on $K$, then $K_\nu$ is a $\d$-algebraic extension of $K$ with the same field of constants as $K$, so maximal Newton-Liouville towers on $K$ exist by Corollary~\ref{cor:cardbound2} and Zorn's lemma.

Let $(K_\mu)_{\mu\leq\nu}$ be a maximal Newton-Liouville tower on $K$ and let $K^{\nl} \coloneqq K_{\nu}$.
By maximality, $K^{\nl}$ is newtonian, and we claim that it is Liouville closed as well.
Lemma~\ref{lem:newtonint} tells us that $K^{\nl}$ is closed under integration.
To see that $K^{\nl}$ is closed under exponential integration, let $s \in K^{\nl}$ with $s<0$.
By maximality, there must be some $a \in (K^{\nl})^\times$ with $v(s- a^\dagger)> \Psi$.
Then Lemma~\ref{lem:newtonint} gives $f \in K^{\nl}$ with $f^\dagger = s-a^\dagger$, so $(fa)^\dagger = s$.

Let $L$ be a newtonian Liouville closed $\Hp$-field extension of $K$.
The embedding properties in Proposition~\ref{prop:newtonization} and Lemma~\ref{lem:expint} give, for each $\mu\leq \nu$, an $\Hp$-field embedding $\imath_\mu\colon K_\mu\to L$, where $\imath_0$ is the inclusion $K \subseteq L$ and $\imath_\mu$ extends $\imath_\eta$ whenever $\eta< \mu\leq \nu$.
Thus, $\imath_\nu$ is the required embedding of $K^{\nl}$ into $L$ over $K$.
\end{proof}

If $K$ is an $\upo$-free $\Hp$-field, then by Lemma~\ref{lem:altnl}, the underlying $H$-field of the $\Hp$-field extension $K^{\nl}$ constructed above is a Newton-Liouville closure of the $H$-field $K$, as defined in~\cite[Section~14.5]{ADH17}.
That is, $K^{\nl}$ embeds over $K$ into any newtonian Liouville closed $H$-field extension of $K$ (not just the $\Hp$-field extensions).
As any two Newton-Liouville closures of $K$ are isomorphic over $K$ by Fact~\ref{fact:nlexist}, we see that \emph{any} Newton-Liouville closure $L$ of $K$ admits a constant power map that makes $L$ an $\Hp$-field extension of $K$.

Now, we turn to constant field extensions.
Given an ordered field extension $C^*$ of $C$, Propositions~10.5.15 and~10.5.16 in~\cite{ADH17} give a unique ordering, valuation, and derivation on the field $K(C^*)$ such that $K(C^*)$ is an $H$-field extension of $K$ with ordered constant field $C^*$, and we always consider $K(C^*)$ as an $H$-field in this way.
With this valuation, $K(C^*)$ has the same value group as $K$.

\begin{proposition}\label{prop:addconstants}
Let $C^*$ be a real closed ordered field extension of $C$ and let $\exp^*$ be an exponential on $C^*$ extending the induced exponential on $C$.
Then the $H$-field $K(C^*)$ has a unique power closure $K(C^*)\pow$, and there is a unique constant power map on $K(C^*)\pow$ making $K(C^*)\pow$ into an $\Hp$-field extension of $K$ with induced exponential $\exp^*$.
With this constant power map, $K(C^*)\pow$ has the following universal property:\ if $L$ is an $\Hp$-field extension of $K$ and $\imath\colon C^*\to C_L$ is an ordered exponential field embedding over $C$, then there is a unique $\Hp$-field embedding $K(C^*)\pow\to L$ over $K$ extending $\imath$.
\end{proposition}
\begin{proof}
Apply Proposition~\ref{prop:mainprop} with $M = K(C^*)$, taking $(a_i)_{i \in I}$ to be the empty tuple.
For the universal property, we first use the universal property of $K(C^*)$ given in~\cite[Proposition~10.5.16]{ADH17} to get an $H$-field embedding $K(C^*)\to L$ extending $\imath$, and then we use the universal property in Proposition~\ref{prop:mainprop} to extend this to an $\Hp$-field embedding $K(C^*)\pow\to L$.
\end{proof}

\subsection{Constructing an $\upo$-free $\Hp$-field extension}
In this subsection, we will show that any $\Hp$-field can be extended to an $\upo$-free $\Hp$-field.

\begin{lemma}\label{lem:makegap}
Suppose that $K$ has asymptotic integration.
Then $K$ has an $\Hp$-field extension with a gap and with constant field $C$.
\end{lemma}
\begin{proof}
By~\cite[Corollary~11.4.10]{ADH17}, $K$ has a spherically complete immediate $H$-field extension $L$.
Corollary~\ref{cor:extensioncor1} tells us that $L\pow$ is an immediate $\Hp$-field extension of $K$, so $L = L\pow$ by spherical completeness.
Using~\cite[Lemma~11.5.14]{ADH17} and the remarks following that lemma, we find $s \in L$ such that $vf$ is a gap in $L(f)$ for any $f$ in an $H$-field extension of $L$ with $f^\dagger = s$ (in the terminology of~\cite{ADH17}, the element $s$ \emph{creates a gap} over $L$).

By~\cite[Lemma~11.5.10]{ADH17} we have $v(s-b^\dagger) \in \Psi_L^\downarrow$ for each $b \in L^\times$, so we may apply Lemma~\ref{lem:expint} to get an $\Hp$-field $L(f)\pow$ extending $L$ with $f^\dagger = s$.
Then $vf$ is a gap in $L(f)\pow$ by the remarks above and Fact~\ref{fact:closurechar}.
\end{proof}

\begin{lemma}\label{lem:gapgoesup}
Let $s \in K$ and suppose that $vs$ is a gap in $K$.
Then $K$ has a grounded $\Hp$-field extension $K(a)\pow$ with $a \prec 1$, $a' = s$, and $C_{K(a)\pow} = C$.
This extension is unique up to $K$-isomorphism:\ given an $\Hp$-field extension $L$ of $K$ and $f \in \smallo_L$ with $f' = s$, there is a unique $\Hp$-field embedding $K(a)\pow\to L$ that sends $a$ to $f$.
\end{lemma}
\begin{proof}
Using~\cite[Lemma~10.2.1 and Corollary~10.5.10]{ADH17}, we find an $H$-field extension $K(a)$ of $K$ with 
\[
a \prec 1,\qquad a' = s,\qquad \Gamma_{K(a)} = \Gamma \oplus \Z va,\qquad C_{K(a)} = C,
\]
and the desired universal property.
The remark following~\cite[Lemma~10.2.1]{ADH17} tells us that $K(a)$ is grounded.
The lemma follows from Corollary~\ref{cor:extensioncor2}.
\end{proof}

Using~\cite[Lemma~10.2.2 and Corollary~10.5.11]{ADH17} in place of Lemma~10.2.1 and Corollary~10.5.10 in the above lemma gives the following variant:

\begin{lemma}\label{lem:gapgoesdown}
Let $s \in K$ and suppose that $vs$ is a gap in $K$.
Then $K$ has a grounded $\Hp$-field extension $K(b)\pow$ with $b \succ 1$, $b' = s$, and $C_{K(b)\pow} = C$.
This extension is unique up to $K$-isomorphism:\ given an $\Hp$-field extension $L$ of $K$ and $g \in L$ with $g \succ 1$ and $g' = s$, there is a unique $\Hp$-field embedding $K(b)\pow\to L$ that sends $b$ to $g$.
\end{lemma}

\begin{lemma}\label{lem:groundedtoomegafree}
Suppose that $K$ is grounded.
Then $K$ has an $\upo$-free $\Hp$-field extension $K\wpow$ with constant field $C$ that embeds over $K$ into any Liouville closed $\Hp$-field extension of $K$.
\end{lemma}
\begin{proof}
Lemma~11.7.17 in~\cite{ADH17} gives an $\upo$-free $H$-field extension $K_{\upo}$ of $K$ with $C_{K_{\upo}} = C$ that embeds over $K$ into any Liouville closed $H$-field extension of $K$.
From the construction of this extension, we have
\[
\Gamma_{K_{\upo}} = \Gamma\oplus \bigoplus_{n > 0}\Z\beta_n
\]
for some elements $\beta_1,\beta_2,\ldots \in \Gamma_{K_{\upo}}$ with $\Psi< \beta_1^\dagger<\beta_2^\dagger < \cdots$.
We conclude by applying Corollary~\ref{cor:maincor2} with $M = K_{\upo}$.
\end{proof}

The embedding property in the previous lemma can be strengthened:\ let $L$ be an $\Hp$-field extension of $K$ and suppose that for each $f \in L^\times$, there is $y \in L$ with $y' = f^\dagger$.
Then $K\wpow$ embeds over $K$ into $L$ (this uses that $K_{\upo}$ enjoys this stronger embedding property).
However, we will only ever use the lemma in the case that $L$ is Liouville closed.

Combining the previous lemmas, we obtain:
\begin{corollary}\label{cor:upofreeext}
$K$ has an $\upo$-free $\Hp$-field extension with constant field $C$.
\end{corollary}

\subsection{Differentially transcendental $\Hp$-field extensions}
In this subsection, we handle various $\Hp$-field extensions of $K$ determined by cuts in $K$.
We begin with adding a new ``small infinite'' element.

\begin{lemma}\label{lem:endelem}
Let $K$ be an $\upo$-free $\Hp$-field.
Then $K$ has an $\Hp$-field extension $K\langle y \rangle\pow$ where $y>0$ and $\Gamma^<<vy<0$.
This extension is unique up to $K$-isomorphism:\ given an $\Hp$-field extension $L$ of $K$ and $y^* \in L^>$ with $\Gamma^<<vy^*<0$, there is a unique $\Hp$-field embedding $K\langle y \rangle\pow\to L$ that sends $y$ to $y^*$.
Additionally, the extension $K\langle y \rangle\pow$ is grounded with $vy^\dagger = \max \Psi_{K\langle y \rangle\pow}$.
\end{lemma}
\begin{proof}
Model-theoretic compactness gives us an $H$-field extension $K\langle y \rangle$ of $K$ where $y>0$ and $\Gamma^<<vy<0$.
By~\cite[Corollary~13.6.8]{ADH17}, this extension is unique up to $K$-isomorphism:\ given an $H$-field extension $L$ of $K$ and $y^* \in L^>$ with $\Gamma^<<vy^*<0$, there is a unique $H$-field embedding $K\langle y \rangle\to L$ that sends $y$ to $y^*$.
Moreover, by~\cite[Lemma~13.4.4 and Proposition~13.6.7]{ADH17}, we have
\[
C_{K\langle y \rangle} = C,\qquad \Gamma_{K\langle y \rangle} = \Gamma \oplus \Z vy \oplus \Z vy^\dagger,\qquad \max \Psi_{K\langle y \rangle} = vy^\dagger.
\]
In order to apply Corollary~\ref{cor:maincor}, we need only show that $y^\dagger$ and $y^{\dagger\dagger}$ are $C$-linearly independent over 
\[
(K^\times)^\dagger + \{g \in K\langle y\rangle:vg >\Psi_{K\langle y \rangle}\}.
\]
Let $c,d \in C$ and $f \in K^\times$ with 
\[
v(cy^\dagger + dy^{\dagger\dagger}+ f^\dagger )> \Psi_{K\langle y \rangle}.
\]
Since $vy^\dagger = \max \Psi_{K\langle y \rangle}$, this gives 
\[
v(cy^\dagger + dy^{\dagger\dagger}+ f^\dagger)>vy^\dagger.
\]
If $d \neq 0$, then we may arrange that $d = 1$.
Applying~\cite[Corollary~9.8.6]{ADH17} with $\beta = vy^\dagger$, we get that $\psi(vy^\dagger+\gamma) \in \Psi$ for all $\gamma \in \Gamma$.
Thus, we have
\[
v(y^{\dagger\dagger} + f^\dagger) = v(y^\dagger f)^\dagger = \psi(vy^\dagger+vf) \in \Psi <vy^\dagger.
\]
It follows that $v(cy^\dagger + dy^{\dagger\dagger}+ f^\dagger) < vy^\dagger$, a contradiction.
Thus, $d = 0$, and $v(cy^\dagger +f^\dagger)> vy^\dagger$.
If $c \neq 0$, then $v(cy^\dagger) = vy^\dagger$, so $v(f^\dagger) = vy^\dagger$ as well, contradicting that $vy^\dagger \not\in \Gamma$.
Thus, $c$ must also equal $0$.
The lemma now follows from Corollary~\ref{cor:maincor} and Fact~\ref{fact:psicofinal}.
\end{proof}

Now we establish an analog of the results in~\cite[Section 16.1]{ADH17} for $\Hp$-fields.

\begin{lemma}\label{lem:dtranselem}
Let $K$ be an $\upo$-free newtonian Liouville closed $\Hp$-field, let $K\langle f \rangle$ be an $H$-field extension of $K$, and suppose that $C_{K\langle f \rangle} = C$ and that $K\langle y \rangle$ is not an immediate extension of $K$ for all $y \in K\langle f\rangle\setminus K$.
Then $K\langle f \rangle$ has a unique power closure $K\langle f \rangle\pow$, and there is a constant power map on $K\langle f \rangle\pow$ that makes it an $\Hp$-field extension of $K$.
If $L$ is an $\Hp$-field extension of $K$ and $g \in L$ realizes the same cut as $f$ over $K$, then there is a unique $\Hp$-field embedding $K\langle f \rangle\pow\to L$ that sends $f$ to $g$.
\end{lemma}
\begin{proof}
Lemma~16.1.2 in~\cite{ADH17} gives a sequence $(\beta_n)$ from $\Gamma_{K\langle f \rangle}$ such that 
\[
\Gamma_{K\langle f\rangle} = \Gamma \oplus \bigoplus_n \Z\beta_n,\qquad \beta_0^\dagger <\beta_1^\dagger<\cdots \text{ is an increasing tuple of elements in $\Gamma_{K\langle f\rangle}\setminus \Gamma$}.
\]
We conclude by Corollary~\ref{cor:maincor2} with $M = K\langle f\rangle$.
The universal property for $K\langle f \rangle\pow$ follows from the embedding property in that corollary, together with the universal property for $K\langle f \rangle$ given in~\cite[Proposition~16.1.5]{ADH17}.
\end{proof}

For our proof of local o-minimality, we need the following variant of the above lemma, where we drop some assumptions on $K$ but now assume that our new element is infinite relative to $K$:

\begin{lemma}\label{lem:infiniteelement}
Let $K$ be an $\upo$-free Liouville closed $\Hp$-field.
Then $K$ has an $\upo$-free $\Hp$-field extension $K\langle a \rangle\pow$ with $a>K$.
This extension is unique up to $K$-isomorphism:\ for any element $b>K$ in an $\Hp$-field extension $L$ of $K$, there is a unique $\Hp$-field embedding $K\langle a \rangle\pow\to L$ that sends $a$ to $b$.
\end{lemma}
\begin{proof}
Compactness gives $a$ in an $H$-field extension of $K$ with $a>K$.
Put $\beta_0 = va$ and for each $n$, put $\beta_{n+1}\coloneqq \beta_n^\dagger$.
Lemmas~16.6.9 and~16.6.10 in~\cite{ADH17} tell us that the ordered differential field $K\langle a \rangle$ is an $H$-field extension of $K$ with
\[
C_{K\langle a \rangle} = C,\qquad \Gamma_{K\langle a \rangle} = \Gamma \oplus \bigoplus_{n}\Z \beta_n,\qquad \beta_0<\beta_1< \beta_2<\cdots < \Gamma,
\]
and the proofs of those lemmas give the desired embedding property for $K\langle a \rangle$.
The lemma follows by applying Corollary~\ref{cor:maincor2} with $M = K\langle a\rangle$.
\end{proof}

\section{Model completeness for $\T\pow$}\label{sec:Tpow}
In this section, we put together the previously established embedding lemmas to prove our main results.

\begin{proposition}\label{prop:bigembedding}
Let $K$ and $L$ be $\upo$-free newtonian Liouville closed $\Hp$-fields and assume that the underlying ordered set of $L$ is $\vert K\vert^+$-saturated and the cofinality of $\Gamma_L^<$ is greater than $\vert \Gamma\vert $.
Let $E$ be an $\upo$-free $\Hp$-subfield of $K$ with $C_E = C$ and let $\imath\colon E \to L$ be an $\Hp$-field embedding.
Then $\imath$ extends to an $\Hp$-field embedding $K\to L$.
\end{proposition}
\begin{proof}
We can assume $E\neq K$, in which case it suffices to show that $\imath$ can be extended to an embedding of an $\upo$-free $\Hp$-subfield of $K$ that properly contains $E$.
If $E$ is not also newtonian and Liouville closed, then we use Proposition~\ref{prop:nlclosure} to identify $E^{\nl}$ with an $\Hp$-subfield of $K$ and to extend $\imath$ to an $\Hp$-field embedding $E^{\nl} \to L$.
We assume for the remainder of the proof that $E$ is an $\upo$-free newtonian Liouville closed $\Hp$-subfield of $K$.

Suppose $\Gamma_E^<$ is not cofinal in $\Gamma^<$, so there is $y \in K^>$ with $\Gamma_E^<<vy<0$.
By our cofinality assumption on $\Gamma_L^<$ there is $y^* \in L^>$ with $\Gamma_{\imath(E)}^<<vy^*<0$.
Using Lemma~\ref{lem:endelem}, we extend $\imath$ to an $\Hp$-field embedding $E\langle y \rangle\pow \to L$.
As $E\langle y \rangle\pow$ is grounded, we can use Lemma~\ref{lem:groundedtoomegafree} to further extend $\imath$ to an embedding of an $\upo$-free $\Hp$-subfield of $K$ containing $y$.
We assume for the remainder of the proof that $\Gamma_E^<$ is cofinal in $\Gamma^<$.
Then every $H$-subfield of $K$ containing $E$ is $\upo$-free by Fact~\ref{fact:dalgupofree}.

Next, suppose $E\langle z\rangle$ is an immediate extension of $E$ for some $z \in K\setminus E$.
The saturation assumption on $L$ gives $z^* \in L$ that realizes the $\imath$-image of the cut over $\imath(E)$ that $z$ realizes over $E$.
By Fact~\ref{fact:immextmap}, we may extend $\imath$ to an $H$-field embedding of $E\langle z\rangle$ into $L$ that sends $z$ to $z^*$.
By Corollary~\ref{cor:extensioncor1}, this extends to an $\Hp$-field embedding $E\langle z \rangle\pow \to L$.

Finally, assume $E$ has no proper immediate extensions in $K$.
Take $f \in K\setminus E$ and, using the saturation assumption on $L$, take $g\in L$ that realizes the $\imath$-image of the cut over $\imath(E)$ that $f$ realizes over $E$.
By Lemma~\ref{lem:dtranselem} we may extend $\imath$ to an $\Hp$-field embedding $E\langle f\rangle\pow\to L$.
\end{proof}

Here is a useful consequence of Proposition~\ref{prop:bigembedding}.

\begin{corollary}\label{cor:bigembedding}
Let $K$ and $L$ be as in the statement of Proposition~\ref{prop:bigembedding}.
If both $K$ and $L$ have small derivation, then any ordered exponential field embedding $\imath\colon C\to C_L$ extends to an $\Hp$-field embedding $K\to L$.
The same is true if both $K$ and $L$ have large derivation.
\end{corollary}
\begin{proof}
Note that $C$ is itself an $\Hp$-field with gap $0$.
We first consider the case that $K$ and $L$ have small derivation.
As $K$ is Liouville closed, there is $f \in K$ with $f' = 1$.
Additionally, $f \succ 1$ since $K$ has small derivation.
Since $L$ is also Liouville closed with small derivation, there is $g \in L$ with $g\succ 1$ and $g' =1$.
Applying Lemma~\ref{lem:gapgoesdown} with $s = 1$, we find a grounded $\Hp$-field $E = C\langle b \rangle\pow$ with $b' = 1$ that embeds into both $K$ and $L$.
By Lemma~\ref{lem:groundedtoomegafree}, these embeddings extend to embeddings of an $\upo$-free $\Hp$-field $E\wpow$ into $K$ and $L$.
Identifying the images of $E\wpow$ in $K$ and $L$, we arrange that $K$ and $L$ share a common $\upo$-free $\Hp$-subfield with the same field of constants as $K$.
Then Proposition~\ref{prop:bigembedding} gives an $\Hp$-field embedding $K\to L$.

Now suppose $K$ and $L$ have large derivation.
Again, since $K$ is Liouville closed, there is $y \in K$ with $y' = -1$.
Then $y\preceq 1$ since $K$ has large derivation, so by subtracting a constant from $y$, we may assume that $y\prec 1$.
The same holds for $L$, so we may apply the preceding argument, this time using Lemma~\ref{lem:gapgoesup} in place of Lemma~\ref{lem:gapgoesdown}.
\end{proof}

\subsection{Model completeness and completeness}

To obtain a model completeness result, we need to remove the assumption that $C_E = C$ in Proposition~\ref{prop:bigembedding}.
To do this, we need to impose additional rigidity on the constant fields.

Let $\cL\pow\coloneqq \{0,1,+,-,\cdot,<,\preceq,\der,\rho,\ee\}$, where $<$ and $\preceq$ are binary relation symbols, $\der$ is a unary function symbol, $\rho$ is a binary function symbol, and $\ee$ is a constant symbol.
We view each $\Hp$-field $K$ as an $\cL\pow$-structure by taking $\rho(f,c) = f^c$ for $(f,c) \in K^>\times C$, letting $\rho$ be identically zero outside $K^>\times C$, and interpreting $\ee$ as an element of $C^>$.
We say that the $\Hp$-field $K$, construed in this language, has \textbf{real exponential constant field} if
\[
(C,c\mapsto 2^c,\ee)\equiv (\R,x\mapsto 2^x,\ee).
\]
Consider the following $\cL\pow$-theories:
\begin{enumerate}
\item $T\pow$, whose models are $\Hp$-fields with real exponential constant field.
\item $T^{\upo}\pow$, whose models are $\upo$-free $\Hp$-fields with real exponential constant field.
\item $T^{\nl}\pow$, whose models are $\upo$-free newtonian Liouville closed $\Hp$-fields with real exponential constant field.
\end{enumerate}

\begin{theorem}\label{thm:mainthm}
The $\cL\pow$-theory $T^{\nl}\pow$ is model complete.
It is the model companion of $T\pow$ and the model completion of $T^{\upo}\pow$.
\end{theorem}
\begin{proof}
Let $K$ and $L$ be models of $T^{\nl}\pow$ and assume that the underlying ordered set of $L$ is $\vert K\vert^+$-saturated, the underlying ordered set of $C_L$ is $\vert C\vert^+$-saturated, and the cofinality of $\Gamma_L^<$ is greater than $\vert \Gamma\vert$.
Let $E\models T^{\upo}\pow$ be an $\cL\pow$-substructure of $K$ and let $\imath\colon E\to L$ be an $\cL\pow$-embedding.
If we can extend $\imath$ to an $\cL\pow$-embedding $K\to L$, then Proposition~\ref{prop:nlclosure} and a standard criterion imply that $T^{\nl}\pow$ is the model completion of $T^{\upo}\pow$; see~\cite[p.~viii]{Kap21}.
Together with Corollary~\ref{cor:upofreeext}, this establishes that $T^{\nl}\pow$ is the model companion of $T\pow$.

To extend $\imath$, we first note that $\Th(\R,x\mapsto 2^x,\ee)$ is model complete and o-minimal as a consequence of Wilkie's theorem~\cite{Wil96}.
Indeed, the real exponential is existentially definable in $(\R,x\mapsto 2^x,\ee)$ by
\[
y=\exp(x)\quad\Longleftrightarrow\quad \exists z\big(2^z=\mathrm{e}\text{ and } 2^{zx}=y\big).
\]
Thus, $C_L$ is saturated as an ordered exponential field, and the inclusions $C_E\subseteq C$ and $\imath(C_E) \subseteq C_L$ are elementary.
We may therefore extend $\imath\vert_{C_E}$ to an ordered exponential field embedding $\jmath \colon C\to C_L$.
By Proposition~\ref{prop:addconstants}, there is a unique $\cL\pow$-embedding $E(C)\pow\to L$ that extends both $\imath$ and $\jmath$.
Since $E(C)\pow$ is $\d$-algebraic over $E$, it is $\upo$-free by Fact~\ref{fact:dalgupofree}.
Now apply Proposition~\ref{prop:bigembedding} with $E(C)\pow$ in place of $E$ to obtain the desired embedding $K\to L$.
\end{proof}

We can use Theorem~\ref{thm:mainthm} and Corollary~\ref{cor:bigembedding} to characterize the completions of $T^{\nl}\pow$.
Let $T^{\nl}\dsmall$ be the $\cL\pow$-theory extending $T^{\nl}\pow$ whose models have small derivation and let $T^{\nl}\dlarge$ be the $\cL\pow$-theory extending $T^{\nl}\pow$ whose models have large derivation.

\begin{theorem}\label{thm:completions}
$T^{\nl}\dsmall$ and $T^{\nl}\dlarge$ are the two completions of $T^{\nl}\pow$.
\end{theorem}
\begin{proof}
Let $K,L \models T^{\nl}\dsmall$ and assume that $L$ is $\vert K\vert^+$-saturated.
Then $C_L$ is $\vert C\vert^+$-saturated, so there is an ordered exponential field embedding $\imath\colon C\to C_L$ since $C$ and $C_L$ are elementarily equivalent; see~\cite[Corollary~B.9.5]{ADH17}.
This, in turn, extends to an embedding $\jmath\colon K\to L$ by Corollary~\ref{cor:bigembedding}.
Then $\jmath$ is elementary since $T^{\nl}\pow$ is model complete, so $K$ and $L$ are elementarily equivalent.
This shows that $T^{\nl}\dsmall$ is complete, and the same proof shows that $T^{\nl}\dlarge$ is complete.
\end{proof}

Our axiomatization of $T^{\nl}\pow$ is effective, with the possible exception of the ``real exponential constant field'' axiom.
Decidability of $\Th(\R,x\mapsto 2^x,\ee)$ is equivalent to decidability of $\Th(\R,\exp)$, so we have the following:
\begin{corollary}\label{cor:decidable}
The theory $T^{\nl}\pow$ and its two completions $T^{\nl}\dsmall$ and $T^{\nl}\dlarge$ are decidable relative to $\Th(\R_{\exp})$.
\end{corollary}

Decidability for the theory of $\R_{\exp}$ is closely connected to transcendental number theory and is implied by Schanuel's conjecture; see~\cite{MW96}.

Recall that $\T\pow$ is the expansion of the ordered valued differential field of transseries with the constant power map $(f,r) \mapsto f^r = \exp(r\log f)$ and the distinguished constant $\ee$.
The exponential on $\T$ satisfies the identity $(\exp f)^\dagger = f'$, so $\T\pow$ is an $\Hp$-field with real exponential constant field by Proposition~\ref{prop:exptopow}.
Thus, we have

\begin{corollary}\label{cor:TTnlsmall}
The $\cL\pow$-theory of $\T\pow$ is model complete, and it is completely axiomatized by $T^{\nl}\dsmall$.
\end{corollary}

By Lemma~\ref{lem:Hpsubfield}, any real closed $H$-subfield of $\T$ that contains $\R$ and is closed under powers is an $\Hp$-subfield of $\T\pow$.
Thus, any Liouville closed $\upo$-free newtonian subfield of $\T$ containing $\R$ is an elementary substructure of $\T\pow$.
These substructures include the subfield $\T^{\operatorname{da}}$ of transseries that are $\d$-algebraic over $\Q$, as well as the subfield $\T_{\operatorname{g}}$ of grid-based transseries; see~\cite{vdH06} for more information about $\T_{\operatorname{g}}$.

\subsection{Automorphisms and non-definability of the exponential}
In terms of definability, $\T\pow$ expands the ordered valued differential field $\T$, and it is a reduct of the ordered valued differential \emph{exponential} field $\T_{\exp}$.
Clearly, $\T\pow$ properly expands $\T$, as the constant field in $\T$ is stably embedded as a real closed ordered field~\cite[Proposition 16.6.7]{ADH17} and $\T\pow$ defines the real exponential on the constants.
Below, we extend this somewhat by showing that there are automorphisms of $\T$ that fix the constants and do not preserve the constant power map.
Additionally, we demonstrate that $\T\pow$ is a strict reduct of $\T_{\exp}$ by identifying an automorphism of $\T\pow$ that does not preserve the exponential.
To do this, we need the following (unpublished) result of Aschenbrenner, van den Dries, and van der Hoeven on strongly $\R$-linear automorphisms of $\T$ (where an automorphism $\sigma \colon \T\to \T$ is said to be \emph{strongly $\R$-linear} if it fixes $\R$ and commutes with infinite sums).

\begin{fact}\label{fact:ADHaut}
Let $\alpha\colon \T\to \R$ be an additive map that vanishes on the convex hull of $\R$.
Then there is a unique strongly $\R$-linear $H$-field automorphism $\sigma_\alpha$ of $\T$ such that
\[
\sigma_\alpha(x) = x,\qquad \sigma_\alpha(\exp f) = \exp(\sigma_\alpha(f) + \alpha(f))\text{ for all }f \in \T,
\]
where $x \in \T$ is the distinguished integral of $1$.
\end{fact}

\begin{lemma}\label{lem:poweraut}
Let $\alpha$ be as above.
Then for $f \in \T^>$ and $r \in \R$, we have $\sigma_\alpha(f^r) = \sigma_\alpha(f)^r$ if and only if $\alpha(r\log f) = r\alpha(\log f)$.
In particular, $\sigma_\alpha$ is an $\Hp$-field automorphism of $\T\pow$ if and only if $\alpha$ is $\R$-linear.
\end{lemma}
\begin{proof}
Let $f \in \T^>$, and let $r \in \R$.
Writing $f = \exp(\log f)$, we get $\sigma_\alpha(f) = \exp(\sigma_\alpha(\log f) + \alpha(\log f))$.
Thus, we have
\begin{align*}
\sigma_\alpha(f^r) &= \sigma_\alpha(\exp(r\log f)) = \exp(\sigma_\alpha(r\log f) + \alpha(r\log f)),\\
\sigma_\alpha(f)^r &= \exp(r\log \sigma_\alpha(f)) = \exp(r\sigma_\alpha(\log f)+ r\alpha(\log f)).
\end{align*}
Since $\sigma_\alpha$ is $\R$-linear, we have $\sigma_\alpha(f^r) = \sigma_\alpha(f)^r$ if and only if $\alpha(r\log f) = r\alpha(\log f)$.
\end{proof}

\begin{corollary}\label{cor:nondef}\
\begin{enumerate}
\item\label{nondef1} Let $f \in \T^>$ be either infinite or infinitesimal and let $r \in \R$ be irrational.
Then there is a strongly $\R$-linear $H$-field automorphism $\sigma\colon \T\to\T$ such that $\sigma(f^r) \neq \sigma(f)^r$.
\item\label{nondef2} Let $g \in \T$ be infinite.
Then there is a strongly $\R$-linear $\Hp$-field automorphism $\tau\colon \T\pow\to\T\pow$ such that $\tau(\exp g) \neq \exp \tau(g)$.
\end{enumerate}
\end{corollary}
\begin{proof}
For~\eqref{nondef1}, note that $\log f$ is infinite, so we may take a $\Q$-vector space decomposition $\T = \Q\log(f) \oplus V$ where $V$ contains both the convex hull of $\R$ in $\T$ and the element $r \log f$.
Let $\alpha \colon \T \to \Q$ be the projection map given by this decomposition, so $\alpha(q\log f+v) = q$ for $q \in \Q$ and $v \in V$.
Then $r \alpha(\log f) = r$, whereas $\alpha(r\log f) = 0$, so Lemma~\ref{lem:poweraut} gives $\sigma_\alpha(f^r) \neq \sigma_\alpha(f)^r$.

For~\eqref{nondef2}, take an $\R$-vector space decomposition $\T = \R g\oplus W$, where $W$ contains the convex hull of $\R$ in $\T$.
Let $\beta\colon \T\to \R$ be the projection map given by this decomposition, so $\beta(rg+w) = r$ for $r \in \R$ and $w \in W$.
As $\beta$ is $\R$-linear, Lemma~\ref{lem:poweraut} gives that $\sigma_\beta$ is an $\Hp$-field automorphism.
It remains to note that 
\[
\sigma_\beta(\exp g) = \exp(\sigma_\beta(g) +1)\neq \exp(\sigma_\beta(g)).\qedhere
\]
\end{proof}

Corollary~\ref{cor:nondef} shows that no restriction of the constant power map to the positive infinite or infinitesimal elements of $\T$ is definable in $\T$, nor is any restriction of the exponential to the infinite elements of $\T$ definable in $\T\pow$.
The restriction of $\exp$ to $\R$ is, of course, definable in $\T\pow$, and the restriction of $\exp$ to the maximal ideal of $\T$ is even definable in $\T$.
Indeed, for $\epsilon \in \T$ with $\epsilon \prec 1$, the element $\exp(\epsilon)$ is the unique solution to the differential equation $y' = \epsilon' y$ with $y\sim 1$; see~\cite[Lemma 8.8]{Kap21}.
Thus, the restriction of $\exp$ to the convex hull of $\R$ is definable in $\T\pow$.

\subsection{Local o-minimality and induced structure on the constants}
In this subsection, let $K \models T^{\nl}\pow$, and fix a $\vert K\vert^+$-saturated elementary extension $K^*$ of $K$.
Here, we describe the structure of unary definable subsets of $K$, as well as the induced structure on the constant field.

\begin{corollary}
Let $X \subseteq K$ be $\cL\pow(K)$-definable.
\begin{enumerate}
\item\label{lomin1} There is $f \in K$ such that $(f,+\infty)$ is either contained in or disjoint from $X$.
\item\label{lomin2} There is $g \in K$ with $g>C$ such that $\{y \in K:C<y<g\}$ is either contained in or disjoint from $X$.
\item\label{lomin3} For each $b \in K$, there is $h \in K$ with $h>b$ such that $(b,h)$ is either contained in or disjoint from $X$.
\end{enumerate}
\end{corollary}
\begin{proof}
Fix elements $a_1,a_2 \in K^*$.
For~\eqref{lomin1}, suppose that $a_1,a_2>K$.
By a standard model-theoretic argument, it suffices to show that $a_1$ and $a_2$ have the same $\cL\pow$-type over $K$.
Lemma~\ref{lem:infiniteelement} tells us that $K\langle a_i \rangle\pow\models T^{\upo}\pow$ for $i = 1,2$ and that there is an $\Hp$-field isomorphism $K\langle a_1 \rangle\pow\simeq K\langle a_2\rangle\pow$ over $K$ that sends $a_1$ to $a_2$.
As $T^{\nl}\pow$ is the model completion of $T^{\upo}\pow$ by Theorem~\ref{thm:mainthm}, this isomorphism is a partial $\cL\pow(K)$-elementary map $K^*\to K^*$.
Note that~\eqref{lomin3} follows from~\eqref{lomin1}, using fractional linear transformations.

For~\eqref{lomin2}, we now assume that $C<a_1,a_2<\{y \in K:y>C\}$.
We again need to show that $a_1$ and $a_2$ have the same $\cL\pow$-type over $K$.
This time, we use Lemma~\ref{lem:endelem} to get an $\Hp$-field isomorphism $K\langle a_1 \rangle\pow\simeq K\langle a_2\rangle\pow$ over $K$ that sends $a_1$ to $a_2$.
Using Lemma~\ref{lem:groundedtoomegafree}, this extends to an $\Hp$-field isomorphism $K_1\simeq K_2$ over $K$, where $K\langle a_i\rangle\pow \subseteq K_i \models T^{\upo}\pow$ for $i = 1,2$.
Again, this isomorphism is a partial $\cL\pow(K)$-elementary map $K^*\to K^*$ by Theorem~\ref{thm:mainthm}.
\end{proof}

All three properties above were established in much the same way for $T^{\nl}$ in~\cite[Section 16.6]{ADH17}.

\begin{corollary}\label{cor:stablyembedded}
The constant field $C$ of $K$ is stably embedded as an ordered exponential field.
That is, any $\cL\pow(K)$-definable subset $X \subseteq C^n$ is definable in $(C,c\mapsto 2^c)$.
\end{corollary}
\begin{proof}
Let $\cL_{\exp}$ be the language of ordered exponential fields, and fix tuples $c_1,c_2 \in C_{K^*}^n$ with the same $\cL_{\exp}$-type over $C$.
By a standard model-theoretic argument, it suffices to show that $c_1$ and $c_2$ have the same $\cL\pow$-type over $K$.
For each $i$, let $C_i$ be the $\cL_{\exp}$-definable closure of $C(c_i)$ in $C_{K^*}$, so $C_i$ is an $\cL_{\exp}$-elementary extension of $C$ and there is an ordered exponential field isomorphism $C_1\to C_2$ over $C$ that sends $c_1$ to $c_2$.
This extends by Proposition~\ref{prop:addconstants} to an $\Hp$-field isomorphism $K(C_1)\pow\to K(C_2)\pow$ over $K$.
This isomorphism is a partial $\cL\pow(K)$-elementary map $K^*\to K^*$ by Theorem~\ref{thm:mainthm}, since each $K(C_i)\pow$ is $\upo$-free by Fact~\ref{fact:dalgupofree}.
\end{proof}

In $T^{\nl}$, the constant field is stably embedded as a real closed ordered field~\cite[Proposition 16.6.7]{ADH17}.
That result was established using quantifier elimination, as opposed to our method above using constant field extensions.

\subsection{Non-uniform constant powers}\label{subsec:nonuni}
Let $K \models T^{\nl}\pow$, and let $\cL\Cpow$ extend the language of ordered rings by constant symbols for each element of the constant field $C$ along with unary function symbols for raising to each of these powers.
That is,
\[
\cL\Cpow \coloneqq\{0,1,+,-,\cdot,<,(c)_{c\in C},(x\mapsto x^c)_{c\in C}\}.
\]
Let $\cL\Cpow^*\coloneqq\cL\Cpow\cup\{\preceq,\der\}$.
We construe $K$ as an $\cL\Cpow^*$-structure in the natural way.
In doing so, we essentially forget the uniformity in our constant power map.
Note that $C$ is an $\cL\Cpow$-substructure of $K$.

Let $T$ denote the complete $\cL\Cpow$-theory of $C$.
As $(C,c\mapsto 2^c)$ is a model of the o-minimal theory $\Th(\R,x\mapsto 2^x)$, the theory of the reduct $T$ is also o-minimal.

\begin{proposition}
The $\cL\Cpow^*$-structure $K$ is an $H_T$-field in the sense of~\cite{Kap23}.
That is, the reduct of $K$ to the language $\cL\Cpow$ is an elementary extension of $C$, and the derivation on $K$ is a $T$-derivation in the sense of~\cite{FK21}.
\end{proposition}
\begin{proof}
Let us first prove this proposition when $K = \T\pow$.
In this case, $C = \R$, and $T$ is the theory of the real field with real powers studied by Miller~\cite{Mil94B}.
Let $T_{\exp}$ be the elementary theory of $\R_{\exp}$ in the language of ordered exponential fields with a constant symbol for each real number.
The underlying exponential field of $\T_{\exp}$ is a model of $T_{\exp}$ by~\cite[Corollary 2.8]{DMM97}, and the derivation on $\T_{\exp}$ is a $T_{\exp}$-derivation; see~\cite[Example 2.16]{FK21}.
As $T$ is a reduct of $T_{\exp}$, this shows that the $\cL\Rpow$-reduct of $\T\pow$ is a model of $T$ and that the derivation on $\T\pow$ is a $T$-derivation.

We now show that the proposition holds whenever $K \models T^{\nl}\dsmall$.
In this case, $K$ is elementarily equivalent to $\T\pow$.
The statements that $K$ is an $\cL\Cpow$-elementary extension of $C$ and that the derivation on $K$ is a $T$-derivation can be expressed by a collection of $\cL\pow$-sentences:\ first rewrite them as $\cL\pow(C)$-sentences in the obvious way, and then assert that they hold for all tuples of constants.
Thus, the proposition holds for $K$ since it holds for $\T\pow$.

Finally, if $K\models T^{\nl}\dlarge$, then by replacing the derivation $\der$ by a compositional conjugate $\phi\inv\der$ for some suitably chosen $\phi \in K^>$, we obtain a model $K^\phi \models T^{\nl}\dsmall$.
Then the proposition holds for $K^\phi$, so it holds for $K$ as well, as being a $T$-derivation is invariant under compositional conjugation; see~\cite[Section 2.2]{Kap22}.
\end{proof}

\section{Hardy fields and surreal numbers}\label{sec:Hardy}
In this section, we show that the results established for Hardy fields and surreal numbers in~\cite{ADH19, ADH24} continue to hold when these structures are expanded by their natural constant power maps.

\subsection{Hardy fields and transfer theorems}

Recall that any Hardy field containing $\R$ is an $H$-field with small derivation and constant field $\R$.

\begin{fact}[{\cite[Chapitre~V]{Bou76}}]\label{fact:LiouHardy}
Any maximal Hardy field is Liouville closed:\ it is real closed, contains $\R$, and is closed under taking integrals, exponentials, and logarithms of positive elements.
\end{fact}

\begin{fact}[\cite{ADH24}]\label{fact:uponewtHardy}
Any maximal Hardy field is $\upo$-free and newtonian.
\end{fact}

By Fact~\ref{fact:LiouHardy}, any maximal Hardy field $M$ is an $H$-field with an exponential $\exp$ that sends the germ $f$ to the germ $\exp f$.
This exponential extends the real exponential on $\R$, and it satisfies the identity $(\exp f)^\dagger = f'$.
Combining this with Proposition~\ref{prop:exptopow}, Fact~\ref{fact:uponewtHardy}, and Lemma~\ref{lem:Hpsubfield}, we obtain the following.

\begin{corollary}\label{cor:HardyTnlsmall}
Any maximal Hardy field, equipped with the constant power map $(f,r) \mapsto f^r = \exp(r\log f)$ and the distinguished constant $\ee$, is a model of $T^{\nl}\dsmall$.
Let $H\supseteq \R$ be a real closed Hardy field and suppose that $H$ is closed under powers (as an $H$-field).
Then $H$ is closed under this constant power map, and with this map, $H$ is a model of $T\pow$.
\end{corollary}

\begin{theorem}\label{thm:transfer}
Let $H\supseteq \R$ be a real closed Hardy field that is closed under powers, and let $\imath\colon H\to \T\pow$ be an $\Hp$-field embedding.
Let $\sigma$ be an $\cL\pow(H)$-sentence, and let $\imath(\sigma)$ be the $\cL\pow(\T)$-sentence obtained by replacing all parameters in $\sigma$ with their images under $\imath$.
The following are equivalent:
\begin{enumerate}
\item\label{transfer1} $M \models \sigma$ for some maximal Hardy field $M \supseteq H$;
\item\label{transfer2} $M\models \sigma$ for every maximal Hardy field $M \supseteq H$;
\item\label{transfer3} $\T\pow\models \imath(\sigma)$.
\end{enumerate}
\end{theorem}
\begin{proof}
We fix a maximal Hardy field $M \supseteq H$, and we view $M$ as a model of $T^{\nl}\dsmall$.
We need to show that $M\models \sigma$ if and only if $\T\pow\models \imath(\sigma)$.
If $H$ is $\upo$-free, then this follows from Theorem~\ref{thm:mainthm}.
If $H$ is grounded, then we may use Lemma~\ref{lem:groundedtoomegafree} to replace $H$ by an $\upo$-free $\Hp$-field $H\wpow$, thereby reducing to the case that $H$ is $\upo$-free.
If $H = \R$, then this follows from Corollary~\ref{cor:bigembedding}:\ take an $\vert M\vert^+$-saturated elementary $\Hp$-field extension $L \supseteq \T\pow$ and extend $\imath\colon \R\to L$ to an $\Hp$-field embedding $M\to L$.
This embedding is elementary, since $T^{\nl}\pow$ is model complete, so $M$ and $L$ (and therefore $M$ and $\T\pow$) have the same $\cL\pow(H)$-theory.

Suppose now that $H \supsetneq \R$ and that $H$ is ungrounded.
We will show that $H$ must be $\upo$-free (so the above cases are exhaustive).
As $\T\pow$ is $\upo$-free, it is enough by Fact~\ref{fact:dalgupofree} to show that $\imath(\Gamma_H^<)$ is cofinal in $\Gamma_{\T}^<$ (where we write $\imath\colon \Gamma_H\to \Gamma_{\T}$ for the map induced by $\imath\colon H\to \T\pow$).
Let $f \in H^>$ with $vf<0$, and put $g \coloneqq \imath(f)$.
Using the fact that $H$ is ungrounded, we define a sequence $(\gamma_n)$ in $\Gamma_H^<$ recursively as follows:
\[
\gamma_0 \coloneqq vf,\qquad \gamma_{n+1}' = \gamma_n^\dagger.
\]
Then $v(\log_ng) = \imath(\gamma_n)$ for each $n$, where $\log_ng$ denotes the $n$-fold iterated logarithm of $g$.
It is well known that $(\log_ng)$ is coinitial among the positive infinite elements of $\T$, so $(\imath(\gamma_n))$ is cofinal in $\Gamma_{\T}^<$.
\end{proof}

Our proof of Theorem~\ref{thm:transfer} relies on the fact that any $H$-subfield of $\T$ that properly contains $\R$ is either grounded or $\upo$-free, a dichotomy that does not apply to all Hardy fields.
Given a real closed Hardy field $H\supseteq \R$ that is closed under powers, one may ask whether conditions~\eqref{transfer1} and~\eqref{transfer2} in Theorem~\ref{thm:transfer} are always equivalent, even when $H$ does not embed into $\T\pow$.
That is, do all maximal Hardy field extensions of $H$ have the same $\cL\pow(H)$-theory? It seems plausible that one can show that this is indeed the case by arguing along the lines of~\cite[Theorem 12.3]{ADH24}.
This uses the formalism of $\Upl\Upo$-fields from~\cite[Chapter 16]{ADH17} and the key fact that Hardy fields admit canonical $\Upl\Upo$-expansions; see~\cite[Lemma 12.1]{ADH24}.

\subsection{Surreal numbers}
The field $\No$ of surreal numbers is a real closed field extension of $\R$ introduced by Conway~\cite{Con76}.
The surreals may be defined in several equivalent ways, but for our purposes, we define a surreal number to be a map $a\colon\gamma\to \{-,+\}$, where $\gamma$ is an ordinal.
For such $a$, this ordinal $\gamma$ is called the \textbf{length of $a$} (sometimes also called the \emph{tree-rank} or \emph{birthday} of $a$).
The collection of all surreal numbers is a proper class, and each ordinal $\gamma$ is identified with the surreal number of length $\gamma$ that takes constant value $+$.
For each $\gamma$, we let $\No(\gamma)$ be the set of surreal numbers of length $<\gamma$.

The surreals admit an exponential, defined by Kruskal and Gonshor~\cite{Gon86}, and with this exponential, $\No$ is an elementary extension of $\R_{\exp}$~\cite{DE01}.
More recently, Berarducci and Mantova equipped the surreals with a derivation that makes $\No$ a Liouville closed $H$-field with constant field $\R$ and that satisfies the identity $\der \exp(a) = \exp(a)\der a$.
We let $\No\pow$ denote the expansion of $\No$ by the Berarducci--Mantova derivation, the constant power map $(f,r)\mapsto \exp(r\log f)\colon \No^>\times \R\to\No^>$, and the distinguished constant $\ee$, so $\No\pow\models T\pow$ by Proposition~\ref{prop:exptopow}.
The $H$-field $\No$ was shown to be $\upo$-free and newtonian in~\cite{ADH19}, so $\No\pow$ is even a model of $T^{\nl}\dsmall$.

Let $\kappa$ be a regular uncountable cardinal.
Then the set $\No(\kappa)$ is an $\upo$-free newtonian Liouville closed $H$-subfield of $\No$ containing $\R$~\cite[Corollary 4.6]{ADH19}.
Lemma~\ref{lem:Hpsubfield} gives that $\No(\kappa)$ is an elementary $\cL\pow$-substructure of $\No\pow$.
The next proposition, an analog of~\cite[Theorem 3]{ADH19}, shows that the surreal numbers are \emph{universal} among models of $T\pow$ with small derivation and archimedean constant field.

\begin{proposition}\label{prop:Noembedding}
Let $K$ be a set-sized $\Hp$-field with small derivation and suppose that $(C,c\mapsto 2^c)$ is an archimedean model of $\Th(\R,x\mapsto 2^x)$.
Then $K$ admits an $\Hp$-field embedding into $\No\pow$.
\end{proposition}
\begin{proof}
It suffices to show that some $\Hp$-field extension of $K$ admits an $\Hp$-field embedding into $\No\pow$.
Since $K$ has small derivation, either $0 \in \Psi^\downarrow$ or $0$ is a gap in $K$.
In the case that $0$ is a gap, we can use Lemma~\ref{lem:gapgoesdown} with $s = 1$ to extend $K$ to a grounded $\Hp$-field $M$ with $0 \in (\Gamma_M^<)' \subseteq \Psi_M^\downarrow$ and $C_M = C$.
Replacing $K$ with $M$, we may assume that $0 \in \Psi^\downarrow$, so any $\Hp$-field extension of $K$ has small derivation.
By Corollary~\ref{cor:upofreeext} and Proposition~\ref{prop:nlclosure}, $K$ has a Liouville closed $\upo$-free newtonian $\Hp$-field extension with the same constant field as $K$, so we may assume that $K \models T^{\nl}\dsmall$.
Let $\kappa\coloneqq \vert K\vert^+$.
Then $\No(\kappa)\models T^{\nl}\dsmall$, and by~\cite[Lemma 5.3]{ADH19}, the underlying ordered sets of $\No(\kappa)$ and $\Gamma_{\No(\kappa)}$ are $\kappa$-saturated.
Since $(C,c\mapsto 2^c)$ is an archimedean model of $\Th(\R,x\mapsto 2^x)$, we obtain an elementary embedding $(C,c\mapsto 2^c)\to (\R,x\mapsto 2^x)$ by the Laskowski--Steinhorn theorem~\cite{LS95}.
By Corollary~\ref{cor:bigembedding} with $\No(\kappa)$ in place of $L$, this embedding extends to an $\Hp$-field embedding $K\to\No(\kappa)$.
\end{proof}

\begin{corollary}
Every real closed Hardy field $H\supseteq \R$ that is closed under powers admits an $\Hp$-field embedding into the surreal numbers over $\R$.
\end{corollary}

There is a natural ordered exponential field embedding $\imath\colon \T\to \No$, which was shown to be an elementary $H$-field embedding in~\cite{ADH19}.
As $\imath$ respects the constant power map, it is even an elementary $\Hp$-field embedding $\T\pow\to \No\pow$.

\bibliographystyle{amsplain}
\bibliography{Constant_powers_biblio}

\end{document}